\documentclass[12pt]{amsart}
\usepackage{amsmath,amsfonts,amstext,amsthm,epsfig,mathrsfs}
\usepackage{amssymb}
\usepackage{MnSymbol}

\newtheorem{Theorem}{Theorem}[section]
\newtheorem{Lemma}[Theorem]{Lemma}
\newtheorem{Corol}[Theorem]{Corollary}
\newtheorem{Prop}[Theorem]{Proposition}
\newtheorem{Rem}[Theorem]{Remark}
\newtheorem{Hyp}[Theorem]{Hypothesis}
\makeatletter
    \addtocounter{section}{0}
    
     \@addtoreset{equation}{section}
\makeatother

\def\N{\mathbb N}
\def\R{\mathbb R}

\def\C{\mathbb C}
\def\loc{\mathrm{loc}}
\def\span{\mathrm{span}}
\newcommand{\weakto}{\rightharpoonup}
\newcommand{\umlaut}{\"}
\def\cl{\overline}

\def\circantiorario{\rcirclerightint}

\DeclareMathOperator*{\essinf}{ess\, inf}

\DeclareMathOperator*{\esslim}{ess\, lim}

\begin{document}

\title[wave equations with steep potential well]
{Singular convergence for semilinear wave equations with steep potential well}%
\author{Martino Prizzi}

\address{Martino Prizzi, Universit\`a di Trieste, Dipartimento di
Matematica, Informatica e Geoscienze, Via Valerio 12/1, 34127 Trieste, Italy}%
\email{mprizzi@units.it}%

\begin{abstract} 
We consider a semilinear wave equation in the whole space with a deep potential well. We prove that as the depth of the well tends to infinity, the solutions of the equation converge to the solutions of a wave equation defined on the bottom of the well, with Dirichlet condition on the boundary.
\end{abstract}
\maketitle


\section{Introduction}
In this paper we consider the semilinear wave equation
\begin{equation}\label{introequation}
\begin{aligned}
u_{tt}+\gamma u_t+ V_\beta(x)u-\Delta u&=f_\beta(x,u),&&(t,x)\in[0,+\infty[\times\R^3
\end{aligned}\end{equation}
in the energy space $H^1(\R^3)\times L^2(\R^3)$. 
The damping coefficient $\gamma$ is a non-negative (possibly zero) constant and the function $V_\beta(x)$ has the form $V_\beta(x)=1+\beta V(x)$, where $\beta$ is a (large) positive parameter and $V(x)$ is a ``well potential''. This means that there exists a bounded open domain 
$\Omega$ on which $V(x)\equiv 0$, while $V(x)$ is strictly positive on $\R^3\setminus\cl\Omega$ and $\lim_{|x|\to\infty} V(x)=1$. The functions $f_\beta(x,u)$ are $C^1$ with respect to $u$ and satisfy a growth condition of type $|\partial_u f_\beta(x,u)|\leq C(a(x)+|u|^2)$, so they define locally Lipschitz continuous mappings from $H^1(\R^3)$ to $L^2(\R^3)$; moreover, as $\beta\to+\infty$, the family $f_\beta$ converges in some sense to a function $f_\Omega$ with analogous properties. It is well known that the Cauchy problem for  (\ref{introequation}) is well posed in $H^1(\R^3)\times L^2(\R^3)$. Our aim is to investigate the behaviour of the solutions when the parameter $\beta$ tends to infinity. There are several results concerning the stationary solutions of (\ref{introequation}) (see e.g. \cite{BPW,BP, GuoTang,J,Wang-Zhou}) and the eigenvales of the linear operator $-\Delta+V_\beta(x)$ (see e.g. \cite{BPW,Liu-Huang,SZ}). Both physical and mathematical  evidence confirm that as $\beta\to+\infty$ the solutions of the nonlinear elliptic equation
\begin{equation}
-\Delta u +V_\beta(x)u=f_\beta(x,u),\quad x\in\R^3
\end{equation} 
 tend to remain confined within the region $\Omega$, and become closer and closer to solutions of the elliptic Dirichlet problem
\begin{equation}
\begin{aligned}-\Delta u+u&=f_\Omega(x,u), && x\in\Omega\\u&=0,&&x\in\partial\Omega\end{aligned}
\end{equation} 
Our aim is to prove that the same is true for non stationary solutions, and to provide a framework to study the (singular) convergence of the dynamics generated by (\ref{introequation}) in the space $H^1(\R^3)\times L^2(\R^3)$ to the one generated by the ``limit'' equation
\begin{equation}\label{introequation2}
\begin{aligned}
u_{tt}+\gamma u_t+u-\Delta u&=f_\Omega(x,u),&&(t,x)\in[0,+\infty[\times\Omega\\u&=0,&&(t,x)\in[0,+\infty[\times\partial\Omega
\end{aligned}
\end{equation}
in the space $H^1_0(\Omega)\times L^2(\Omega)$. In the spirit of works like \cite{ACL, BCCD, CR,PR01,PR03} and others, we have in mind the persistence of attractors or more general invariant sets, the robustness of exponential dichotomies, and other dynamical aspects of equation (\ref{introequation}). 

The paper is organized as follows:

\smallskip

In Section 2 we introduce the functional analytic setting for the linear elliptic problem and we prove the (singular) convergence of the resolvent of the operator $-\Delta+V_\beta(x)$ in $\R^3$ to the resolvent of the operator $-\Delta +I$ in $\Omega$ (with Dirichlet boundary condition) when $\beta\to+\infty$.

\smallskip

In Section 3 we study the behaviour of the spectrum of $-\Delta+V_\beta(x)$ as $\beta\to+\infty$. We show that as $\beta$ increases, the bottom of the essential spectrum moves towards $+\infty$ and more and more eigenvalues appear. The eigenvalues converge to the ones of the limit operator $-\Delta+I$ in $\Omega$ (with Dirichlet boundary condition), and so do the corresponding eigenfunctions. This spectral convergence can be used e.g. in studying exponential dichotomies for the linearization of (\ref{introequation}) near an equilibrium.

\smallskip

In Section 4 we introduce the functional analytic setting for the linear wave equation
\begin{equation}
\begin{aligned}
u_{tt}+\gamma u_t+ V_\beta(x)u-\Delta u&=0,&&(t,x)\in[0,+\infty[\times\R^3
\end{aligned}\end{equation}
and we prove that such equation generates a strongly continuous semigroup $T_\beta(t)$ in the space $H^1(\R^N)\times L^2(\R^3)$, with bounds independent of $\beta$.

\smallskip

In Section 5 we prove a singular Trotter-Kato theorem for the semigroups $T_\beta(t)$ as $\beta\to+\infty$, showing that they converge to the semigroup $T_\Omega(t)$ generated by the linear wave equation
\begin{equation}
\begin{aligned}
u_{tt}+\gamma u_t+u-\Delta u&=0,&&(t,x)\in[0,+\infty[\times\Omega\\u&=0,&&(t,x)\in[0,+\infty[\times\partial\Omega
\end{aligned}
\end{equation}

\smallskip

In Section 6 we finally prove that, as $\beta\to+\infty$, the solutions of the semilinear equation (\ref{introequation}) converge to the solutions of (\ref{introequation2}) uniformly on compact time intervals.


\section{Resolvent convergence} 

In this paper we denote by $|\cdot|$ the euclidean norm in $\R^N$. Let $ \Omega\subset\R^N$ be an open bounded set with $C^2$ -boundary and let $V\colon\R^N\to\R$ be a bounded measurable function. Throughout the paper we make the following assumption:

\begin{Hyp}\label{prop_V}
The function $V$ satisfies the following properties:
\begin{enumerate}
\item $0\leq V(x) \leq 1$ almost everywhere in $\R^N$;
\item $V(x)= 0$ almost everywhere in $\cl\Omega$;
\item $\essinf \{V(x)\mid x\in \cl{B_r(x_0)}\}>0$ for every closed ball $\cl{B_r(x_0)}\subset\R^N\setminus\cl\Omega$; 
\item $\esslim_{|x|\to\infty} V(x)=1$.
\end{enumerate}
\end{Hyp}
For every $\beta\geq0$ we define
\begin{equation}V_\beta(x):=1+\beta V(x),\quad x\in\R^N.\end{equation}

For $\beta\geq0$, in $H^1(\R^N)$ we define the following bilinear form:
\begin{equation}
a_\beta(u,v):=\int_{\R^N}\nabla u(x)\cdot\nabla v(x) \,dx+\int_{\R^N}V_\beta(x)u(x)v(x)\,dx,\quad u,v\in H^1(\R^N).
\end{equation}

It turns out that $a_\beta(\cdot,\cdot)$ is a scalar product in $H^1(\R^N)$, equivalent to the standard one. We shall denote this scalar product by $\langle \cdot,\cdot\rangle_{1,\beta}$. Setting $\|u\|_{1,\beta}:=\langle u,u\rangle_{1,\beta}^{1/2}$, we have:
\begin{equation}
\|u\|_{H^1(\R^N)}\leq\|u\|_{1,\beta}\leq(1+\beta)^{1/2}\|u\|_{H^1(\R^N)}
\end{equation}

The scalar product $\langle \cdot,\cdot\rangle_{1,\beta}$ induces a bounded linear map $\Lambda_\beta\colon H^1(\R^N)\to H^{-1}(\R^N)$ through the assignement
\begin{equation}
\langle \Lambda_\beta u,v\rangle_{H^{-1}(\R^N),H^1(\R^N)}:=\langle u,v\rangle_{1,\beta},\quad u,v\in H^1(\R^N).
\end{equation}
One can easily check that $\|\Lambda_\beta u\|_{-1,\beta}=\|u\|_{1,\beta}$, where $\|\cdot\|_{-1,\beta}$ is the dual norm in $H^{-1}(\R^N)$ relative to the norm $\|\cdot\|_{1,\beta}$ in $H^1(\R^N)$.
It follows from the Lax-Milgram theorem that $\Lambda_\beta$ is invertible. Therefore it is natural to equip $H^{-1}(\R^N)$ with the scalar product 
\begin{equation}
\langle u, v\rangle_{-1,\beta}:=\langle \Lambda_\beta^{-1}u,\Lambda_\beta^{-1} v\rangle_{1,\beta},
\end{equation}
which returns exacly the norm $\|\cdot\|_{-1,\beta}$ and makes $\Lambda_\beta$ an isometry.
The part of $\Lambda_\beta$ in $L^2(\R^N)$ defines a selfadjoint operator in $L^2(\R^N)$, which we denote by $A_\beta$. The domain of $A_\beta$ is $D(A_\beta)=\{u\in H^1(\R^N)\mid \Lambda_\beta u\in L^2(\R^N)\}$. Concretely, the operator $A_\beta$ acts according to the rule
\begin{equation}
\langle A_\beta u,v\rangle_{L^2(\R^N)}=a_\beta(u,v), \quad u\in D(A_\beta),v\in H^1(\R^N).
\end{equation}
By regularity theory for elliptic equations (see e.g. \cite{GilTrud}), it turns out that $D(A_\beta)=H^2(\R^N)$, and
\begin{equation}A_\beta u=-\Delta u+V_\beta(x) u,\quad u\in D(A_\beta).\end{equation}

We introduce the following closed subspace of $H^1(\R^N)$:
\begin{equation}
H^1_\Omega(\R^N):=\{u\in H^1(\R^N)\mid u(x)=0\,\, \text{a.e. in}\,\, \R^N\setminus\Omega\}
\end{equation}
Notice that $\langle u,v\rangle_{1,\beta}=\langle u,v\rangle_{H^1(\R^N)}$ whenever $u,v\in H^1_\Omega(\R^N)$, so $\|u\|_{1,\beta}=\|u\|_{H^1(\R^N)}$ for all $u\in H^1_\Omega(\R^N)$. Moreover, $H^1_\Omega(\R^N)$ can be naturally identified with $H^1_0(\Omega)$. 
We call $L^2_\Omega(\R^N)$ the closure of $H^1_\Omega(\R^N)$ in $L^2(\R^N)$, and we observe that 
\begin{equation}
L^2_\Omega(\R^N):=\{u\in L^2(\R^N)\mid u(x)=0\,\, \text{a.e. in}\,\, \R^N\setminus\Omega\},
\end{equation}
so we can naturally identify $L^2_\Omega(\R^N)$ with $L^2(\Omega)$. We denote by $H^{-1}_\Omega(\R^N)$ the dual space of $H^1_\Omega(\R^N)$ and we observe that $H^{-1}_\Omega(\R^N)$ can be naturally identified with $H^{-1}(\Omega)$.

The scalar product $\langle \cdot,\cdot\rangle_{H^1_\Omega(\R^N)}$ induces a bounded linear map $\Lambda_\Omega\colon H^1_\Omega(\R^N)\to H^{-1}_\Omega(\R^N)$ through the assignement
\begin{equation}
\langle \Lambda_\Omega u,v\rangle_{H^{-1}_\Omega(\R^N),H^1_\Omega(\R^N)}:=\langle u,v\rangle_{H^1_\Omega(\R^N)},\quad u,v\in H^1_\Omega(\R^N).
\end{equation}
One can easily check that $\|\Lambda_\Omega u\|_{H^{-1}_\Omega(\R^N)}=\|u\|_{H^1_\Omega(\R^N)}$, where $\|\cdot\|_{H^{-1}\Omega}$ is the dual norm in $H^{-1}_\Omega(\R^N)$ relative to the norm $\|\cdot\|_{H^1_\Omega(\R^N)}$ in $H^1_\Omega(\R^N)$.
It follows from the Lax-Milgram theorem that $\Lambda_\Omega$ is invertible. Therefore it is natural to equip $H^{-1}_\Omega(\R^N)$ with the scalar product 
\begin{equation}
\langle u, v\rangle_{H^{-1}_\Omega(\R^N)}:=\langle \Lambda_\Omega^{-1}u,\Lambda_\Omega^{-1} v\rangle_{H^1_\Omega(\R^N)},
\end{equation}
which returns exacly the norm $\|\cdot\|_{H^{-1}_\Omega}(\R^N)$ and makes $\Lambda_\Omega$ an isometry.
The part of $\Lambda_\Omega$ in $L^2_\Omega(\R^N)$ defines a selfadjoint operator in $L^2_\Omega(\R^N)$, which we denote by $A_\Omega$. The domain of $A_\Omega$ is $D(A_\Omega)=\{u\in H^1_\Omega(\R^N)\mid \Lambda_\Omega u\in L^2_\Omega(\R^N)\}$. Concretely, the operator $A_\Omega$ acts according to the rule
\begin{equation}
\langle A_\Omega u,v\rangle_{L^2_\Omega(\R^N)}=\langle u,v\rangle_{H^1_\Omega(\R^N)}, \quad u\in D(A_\Omega),v\in H^1_\Omega(\R^N).
\end{equation}
By regularity theory for elliptic equations it turns out that $D(A_\Omega)=\{u\in H^1_\Omega(\R^N)\mid u|_\Omega\in H^2(\Omega)\}$, which can be naturally identified with $H^1_0(\Omega)\cap H^2(\Omega)$, and
\begin{equation}(A_\Omega u) (x)=\begin{cases}-\Delta u (x)+u(x)&\text{$x\in\Omega$,}\\0&\text{$x\in\R^N\setminus\Omega$.}\end{cases}\quad \end{equation}
The operator $A_\Omega$ can be naturally identified with the operator $u\mapsto -\Delta u +u$ from $H^1_0(\Omega)\cap H^2(\Omega)$ to $L^2(\Omega)$. 

A straightforward application of the Lax-Milgram theorem guarantees that for all $\lambda>-1$ both $A_\beta+\lambda I$ and $A_\Omega+\lambda I$ possess a bounded inverse. We have the following convergence result:

\begin{Theorem} Let $(\beta_n)_{n\in\N}$ be a sequence of positive numbers, $\beta_n\to+\infty$ as $n\to\infty$. Let $(f_n)_{n\in\N}$ be a sequence in $L^2(\R^N)$, and assume that $f_n\to f$ strongly in $L^2(\R^N)$ as $n\to\infty$, where $f\in L^2_\Omega(\R^N)$. Let $\lambda>-1$ be fixed, and for every $n\in\N$ let $u_n\in H^1(\R^N)$ be the weak solution of the problem
\begin{equation}
A_{\beta_n}u+\lambda u=f_n.\label{ell1}
\end{equation}
Moreover, let $u_\Omega\in H^1_\Omega(\R^N)$ be the weak solution of the problem
\begin{equation}
A_\Omega u+\lambda u=f.\label{ell2}
\end{equation}
Then, up to a subsequence,
\begin{equation}
\lim_{n\to\infty}\|u_n-u_\Omega\|_{1,\beta_n}=0.\label{convergenza1}
\end{equation}
\end{Theorem}
\begin{proof} For ease of notation, we assume that $\lambda\geq 0$: when $-1<\lambda<0$ only a few minor changes are needed.  For  every $n\in\N$ and $\phi\in H^1(\R^N)$ we have
\begin{equation}
\int_{\R^N}\nabla u_n\cdot\nabla\phi\,dx+(1+\lambda)\int_{\R^N}u_n\phi\,dx+\beta_n\int_{\R^N}V(x)u_n\phi\,dx=\int_{\R^N}f_n\phi\,dx\label{debole}
\end{equation}
Choosing $\phi=u_n$ we get
\begin{equation}
\|u_n\|_{H^1(\R^N)}^2+\lambda\|u_n\|_{L^2(\R^N)}^2+\beta_n \int_{\R^N}V(x)|u_n|^2\,dx\leq \|f_n\|_{L^2(\R^N)}\|u_n\|_{L^2(\R^N)}\label{stima1}
\end{equation}
Setting $K:=\sup_n\|f_n\|_{L^2(\R^N)}$, we obtain that
\begin{equation}
\|u_n\|_{H^1(\R^N)}\leq K\quad\text{for all $n\in\N$.}\end{equation}
It follows by the Banach-Alaoglu theorem that there exists $u_\infty\in H^1(\R^N)$ such that, up to a subsequence,
\begin{equation}
u_n\weakto u_\infty\quad\text{in $H^1(\R^N)$ as $n\to\infty$.}
\end{equation}
By Rellich's theorem we have that
\begin{equation}
u_n\to u_\infty \quad\text{in $L^2_{\loc}(\R^N)$ as $n\to\infty$.}
\end{equation}
By (\ref{stima1}) we get that
\begin{equation}
\beta_n \int_{\R^N}V(x)|u_n|^2\,dx\leq K^2\quad\text{for all $n\in\N$.}
\end{equation}
Let $\cl{B_r(x_0)}$ be a closed ball contained in $\R^N\setminus\cl\Omega$. Then we have
\begin{equation}\left(\essinf_{\cl{B_r(x_0)}} V(x)\right)\int_{\cl{B_r(x_0)}}|u_n|^2\,dx\leq\frac{K^2}{\beta_n}\to 0\quad\text{as $n\to\infty$.}\end{equation}
It follows that $u_\infty(x)=0$ a.e. in $\R^N\setminus\Omega$, i.e. $u_\infty\in H^1_\Omega(\Bbb R^N)$. Now choosing $\phi\in H^1_\Omega(\R^N)$ in (\ref{debole}) we see that the term $\beta_n\int_{\R^N}V(x)u_n\phi\,dx$ disappears and passing to the limit for $n\to\infty$ we  get
\begin{equation}
\int_{\R^N}\nabla u_\infty\cdot\nabla\phi\,dx+(1+\lambda)\int_{\R^N}u_\infty\phi\,dx=\int_{\R^N}f\phi\,dx.
\end{equation}
Therefore $u_\infty$ is a weak solution of (\ref{ell1}), so $u_\infty=u_\Omega$. In order to complete the proof, we need to prove (\ref{convergenza1}). To this end, we observe that
\begin{multline*}
\|u_n-u_\Omega\|_{1,\beta_n}^2+\lambda \|u_n-u_\Omega\|_{L^2(\R^N)}^2\\
=\int_{\R^N}|\nabla(u_n-u_\Omega)|^2\,dx+
(1+\lambda)\int_{\R^N}|u_n-u_\Omega|^2\,dx+\beta_n\int_{\R^N}V(x)|u_n-u_\Omega|^2\,dx\\
=\int_{\R^N}|\nabla u_n|^2\,dx+
(1+\lambda)\int_{\R^N}|u_n|^2\,dx+\beta_n\int_{\R^N}V(x)|u_n|^2\,dx\\
-2\int_{R^N}\nabla u_n\cdot\nabla u_\Omega\,dx-2(1+\lambda)\int_{\R^N}u_nu_\Omega\,dx\\
+\int_{\R^N}|\nabla u_\Omega|^2\,dx+
(1+\lambda)\int_{\R^N}|u_\Omega|^2\,dx\\
=\int_{R^N}f_nu_n\,dx-2\int_{R^N}\nabla u_n\cdot\nabla u_\Omega\,dx-2(1+\lambda)\int_{\R^N}u_nu_\Omega\,dx+\int_{\R^N}fu_\Omega\,dx.
\end{multline*}
Since $u_n\weakto u_\Omega$ in $H^1(\R^N)$ and $f_n\to f$ in $L^2(\R^N)$ as $n\to\infty$, we get
\begin{multline*}
\lim_{n\to\infty}\left(\|u_n-u_\Omega\|_{1,\beta_n}^2+\lambda \|u_n-u_\Omega\|_{L^2(\R^N)}^2\right)\\
=\int_{\R^N} fu_\Omega\,dx-2\left(\int_{\R^N}|\nabla u_\Omega|^2\,dx+(1+\lambda)\int_{\R^N}|u_\Omega|^2\,dx\right)+\int_{\R^N} fu_\Omega\,dx=0,
\end{multline*}
and this completes the proof.
\end{proof}
We can rephrase the above result in the following way:
\begin{Corol}\label{resolvconverg} Let $(\beta_n)_{n\in\N}$ be a sequence of positive numbers, $\beta_n\to+\infty$ as $n\to\infty$. Let $(f_n)_{n\in\N}$ be a sequence in $L^2(\R^N)$, and assume that $f_n\to f$ strongly in $L^2(\R^N)$ as $n\to\infty$, where $f\in L^2_\Omega(\R^N)$. Let $\lambda>-1$ be fixed. Then
\begin{equation}
\|(-A_{\beta_n}-\lambda I)^{-1}f_n-(-A_{\Omega}-\lambda I)^{-1}f\|_{1,\beta_n}\to 0\quad\text{as $n\to\infty$.}
\end{equation}
\end{Corol}\hfill$\square$

Several important questions arising in the study of hyperbolic or parabolic evolution equations involving elliptic operators require a precise knowledge of the placement of their eigenvalues and of the shape of the corresponding eigenfunction. For example, the ``degree of instability'' of a stationary point is measured by the number of positive eigenvalues of the linearized equation; the existence of an exponential dichotomy for a semigroups depends on the placement of the eigenvalues of its generator, etc. With these kind of applications in mind, in the next section we prove a spectral convergence result for the family $(A_\beta)_{\beta\geq 0}$ as $\beta\to\infty$.


\section{Spectral convergence}

Since $\Omega$ is bounded and has regular boundary, Rellich's theorem implies that $A_\Omega$ has compact resolvent and hence its spectrum $\sigma(A_\Omega)$ consists of a non-decreasing sequence  of real eigenvalues with finite multiplicity, diverging to $+\infty$.
Let
\begin{equation}0<\lambda_1\leq\lambda_2\leq\lambda_3\leq\dots\,,\quad \lambda_k\to+\infty \quad\text{as $k\to\infty$,}\end{equation}
be the sequence of the eigenvalues of $A_\Omega$, repeated according to their (finite) multiplicity. For every $k\in\N$ let $\varphi_k\in H^1_\Omega(\R^N)$ be an eigenfunction for the eigenvalue $\lambda_k$, that is
\begin{equation}
A_\Omega \varphi_k=\lambda_k \varphi_k, \quad k\in\N,
\end{equation}
with $\|\varphi_k\|_{L^2_\Omega(\R^N)}=1$ and $\langle \varphi_i,\varphi_j\rangle_{L^2_\Omega(\R^N)}=0$ for $i\not=j$. 
Then (see e.g. \cite[Ch. 6]{R-T}) $(\varphi_k)_{k\in\N}$ is a complete orthonormal system in $L^2_\Omega(\R^N)$ and the following recursive formula holds for $\lambda_k$ and $\varphi_k$:
\begin{multline}
\lambda_k=\int_{\R^N}|\nabla \varphi_k|^2\,dx+\int_{\R^N}|\varphi_k|^2\,dx\\=\min_{\substack{v\in H^1_\Omega,\|v\|_{L^2_\Omega}=1\\ v\perp\span\langle \varphi_1,\dots, \varphi_{k-1}\rangle}}
\left(\int_{\R^N}|\nabla v|^2\,dx+\int_{\R^N}|v|^2\,dx\right),\quad k=1,2,3,\dots\, ,\label{recursive}\end{multline}
where $\perp$ denotes orthogonality with respect to the scalar product in  $L^2_\Omega(\R^N)$. Moreover, whenever $(\tilde\lambda_k)_{k_\N}$ is a sequence of real numbers and $(\tilde \varphi_k)_{k\in\N}$ is a sequence of functions in $H^1_\Omega(\R^N)$, with $\|\tilde \varphi_k\|_{L^2_\Omega(\R^N)}=1$ for all $k\in\N$ and $\langle \tilde \varphi_i,\tilde \varphi_j\rangle_{L^2_\Omega(\R^N)}=0$ for $i\not=j$,  such that the recursive formula (\ref{recursive}) holds with $\lambda_k$ replaced by $\tilde\lambda_k$ and $\varphi_k$ replaced by $\tilde \varphi_k$, then  $\tilde\lambda_k=\lambda_k$ for all $k$, $A_\Omega \tilde \varphi_k=\lambda_k \tilde \varphi_k$ for all $k$, and $(\tilde \varphi_k)_{k\in\N}$ is a complete orthonormal system in $L^2_\Omega(\R^N)$.
The following (non recursive) formulas are also very important:
\begin{equation}
\lambda_k=\max_{\substack{U\subset H^1_\Omega\\\dim U\leq k-1}}\,\min_{\substack{v\in H^1_\Omega,\|v\|_{L^2_\Omega}=1\\ v\perp U}}
\left(\int_{\R^N}|\nabla v|^2\,dx+\int_{\R^N}|v|^2\,dx\right),\quad k=1,2,3,\dots \label{nonrecursive1}
\end{equation}
and
\begin{equation}
\lambda_k=\min_{\substack{U\subset H^1_\Omega\\\dim U=k}}\,\max_{\substack{\|v\|_{L^2_\Omega}=1\\v\in U}}\left(\int_{\R^N}|\nabla v|^2\,dx+\int_{\R^N}|v|^2\,dx\right)\quad k=1,2,3,\dots\label{nonrecursive2}
\end{equation}
where $U$ denotes a subspace of $H^1_\Omega(\R^N)$ and $\perp$ denotes orthogonality with respect to the scalar product in  $L^2_\Omega(\R^N)$.  In (\ref{nonrecursive1}) for every $k$ the maximum is attained at $U=\span\langle \varphi_1,\dots,\varphi_{k-1}\rangle$, and the minimum in $U^\perp$ is attained at $v=\varphi_k$. In (\ref{nonrecursive2}) the minimum is attained at $U=\span\langle \varphi_1,\dots,\varphi_{k}\rangle$, and the maximum within $U$ is attained at $v=\varphi_k$.

The operators $A_\beta$ do not have compact resolvent, so the situation changes drastically. The spectrum $\sigma(A_\beta)$ of $A_\beta$ is contained in $[1,+\infty[$, while the essential spectrum $\sigma_{\rm ess}(A_\beta)$ is contained in $[1+\beta, +\infty[$, due to Weyl's theorem (see \cite[pp. 112-113]{RS-4}). Denoting by $\sigma_{\beta,0}$ the bottom of $\sigma_{\rm ess}(A_\beta)$, the part of the spectrum of $A_\beta$ below $\sigma_{\beta,0}$ consists either of a finite number of eigenvalues with finite multiplicity or of a non-decreasing sequence of eigenvalues with finite multiplicity, converging to the bottom of the essential spectrum. 

The eigenvalues (and the corresponding eigenfunctions) of $A_\beta$ are characterized in the following way (see \cite[Th. XIII.1]{RS-4}):
 \begin{Theorem}\label{charact}For $k=1,2,3,\dots$, define the numbers
\begin{equation}
\lambda_{\beta,k}:=\sup_{\substack{U\subset H^1\\\dim U\leq k-1}}\,\inf_{\substack{v\in H^1,\|v\|_{L^2}=1\\ v\perp U}} \left(\int_{\R^N}|\nabla v|^2\,dx+\int_{\R^N}V_\beta(x)|v|^2\,dx\right),\label{proper1}
\end{equation}
where $U$ denotes a subspace of $H^1(\R^N)$. Then $\lambda_{\beta,1}\leq\lambda_{\beta,2}\leq\lambda_{\beta,3}\leq\dots$ and for every $k\in\N$ one has that either
\begin{itemize}
\item[a)] $\lambda_{\beta,k}<\sigma_{\beta,0}$, and then $A_\beta$ has at least $k$ eigenvalues (counting multiplicity) below $\sigma_{\beta,0}$ and $\lambda_{\beta,k}$ is the $k$-th eigenvalue;
\end{itemize}
or
\begin{itemize}
\item[b)] $\lambda_{\beta,k}=\sigma_{\beta,0}$, and then $\lambda_{\beta,k}=\lambda_{\beta,k+1}=\lambda_{\beta,k+2}=\dots$ and there are at most $k-1$ eigenvalues of $A_\beta$ below $\sigma_{\beta,0}$.
\end{itemize}
\end{Theorem}\hfill$\square$

The following formula also holds (see \cite[Th. XIII.3]{RS-4}):
\begin{equation}
\lambda_{\beta,k}=\inf_{\substack{U\subset H^1\\\dim U=k}}\,\max_{\substack{\|v\|_{L^2}=1\\v\in U}}\left(\int_{\R^N}|\nabla v|^2\,dx+\int_{\R^N}V_\beta(x)|v|^2\,dx\right)\quad k=1,2,3,\dots\label{proper2}
\end{equation}
where $U$ denotes a subspace of $H^1(\R^N)$.

If $\lambda_{\beta,\bar k}<\sigma_{\beta,0}$, then for $k=1$, \dots, $\bar k$, the supremum and the infimum in (\ref{proper1}) are actually a minimum and a maximum, and the following recursive formula holds (see \cite[pp. 241-242]{L-L}):
\begin{multline}
\lambda_{\beta, k}=\int_{\R^N}|\nabla \varphi_{\beta,k}|^2\,dx+\int_{\R^N}V_\beta(x)|\varphi_{\beta,k}|^2\,dx\\=\min_{\substack{v\in H^1,\|v\|_{L^2}=1\\ v\perp\span\langle \varphi_{\beta,1},\dots, \varphi_{\beta,k-1}\rangle}}
\left(\int_{\R^N}|\nabla v|^2\,dx+\int_{\R^N}V_\beta(x)|v|^2\,dx\right),\quad k=1,2,3,\dots,\bar k ,\label{proper-recursive}\end{multline}
where $\perp$ denotes orthogonality with respect to the scalar product in  $L^2(\R^N)$, and for each $k$ the function $\varphi_{\beta,k}$ is an eigenfunction of $\lambda_{\beta,k}$. Moreover, whenever $\tilde\lambda_{\beta,k}$, $k=1$, \dots, $\bar k$, are numbers and $\tilde \varphi_{\beta,k}$, $k=1$, \dots, $\bar k$, are functions in $H^1(\R^N)$ with $\|\tilde \varphi_{\beta,k}\|_{L^2_\Omega(\R^N)}=1$ and $\langle \tilde \varphi_{\beta,i},\tilde \varphi_{\beta,j}\rangle_{L^2_\Omega(\R^N)}=0$ for $i\not=j$, such that the recursive formula (\ref{proper-recursive}) holds with $\lambda_{\beta,k}$ replaced by $\tilde\lambda_{\beta,k}$ and $\varphi_{\beta,k}$ replaced by $\tilde \varphi_{\beta,k}$, then  $\tilde\lambda_{\beta,k}=\lambda_{\beta,k}$ and  $A_\beta \tilde \varphi_{\beta,k}=\lambda_{\beta,k} \tilde \varphi_{\beta,k}$ for $k=1$, \dots, $\bar k$. If $\lambda_{\beta,\cl k}<\lambda_{\beta,\cl k+1}$, then $\span\langle \tilde \varphi_{\beta,1},\dots,\tilde \varphi_{\beta,\bar k}\rangle=\span\langle  \varphi_{\beta,1},\dots,\varphi_{\beta,\bar k}\rangle$.

We have the following spectral convergence result:
\begin{Theorem}\label{convergenzaspettrale} Let $(\beta_n)_{n\in\N}$ be a sequence of positive numbers, $\beta_n\to+\infty$ as $n\to\infty$. Let $\bar k\in\N$ and let $\bar n\in\N$ be such that $\beta_n>\lambda_{\bar k}$ for all $n\geq \bar n$. Then the following statements hold true:
\begin{enumerate}
\item for all $n\geq\bar n$, $\sigma_{\beta_n,0}\geq\lambda_{\bar k}+1$;
\item for all $n\geq\bar n$, $\lambda_{\beta_n,1}$, \dots, $\lambda_{\beta_n,\bar k}$ are eigenvalues of $A_{\beta_n}$, with a corresponding family of eigenfunctions $\varphi_{\beta_n,1}$, \dots, $\varphi_{\beta_n,\bar k}$ satisfying the recursive formula (\ref{proper-recursive});
\item for $k=1$,\dots,$\bar k$, $\lambda_{\beta_n,k}\to\lambda_k$ as $n\to\infty$;
\item there exist $\tilde \varphi_1$, \dots,$\tilde \varphi_{\bar k}\in H^1_\Omega(\R^N)$, with $\|\tilde \varphi_k\|_{L^2_\Omega(\R^N)}=1$ for $k=1$,\dots,$\bar k$ and $\langle \tilde \varphi_i,\tilde \varphi_j\rangle_{L^2_\Omega(\R^N)}=0$ for $i\not=j$, satisfying
the recursive formula (\ref{recursive}) with $\varphi_{k}$ replaced by $\tilde \varphi_{k}$ for $k=1$,\dots,$\bar k$ (so $A_\Omega \tilde \varphi_k=\lambda_k\tilde \varphi_k$), such that, up to a subsequence, $\|\varphi_{\beta_n,k}-\tilde \varphi_k\|_{\beta_n,1}\to 0$ as $n\to\infty$.
\end{enumerate}
\end{Theorem}
\begin{proof}
Statement (1) is trivially true since for $n\geq\bar n$ one has 
$\sigma_{\beta_n,0}\geq 1+\beta_n\geq \lambda_{\bar k}+1$. In order to prove statement (2) we
choose an orthonormal set of functions $\varphi_1$, \dots, $\varphi_{\bar k}\in H^1_\Omega(\R^N)$ such that $A_\Omega \varphi_k=\lambda_k \varphi_k$, $k=1$, \dots, $\bar k$, and then we apply (\ref{proper2}) and (\ref{nonrecursive2}):
\begin{multline}
\lambda_{\beta_n,k}=\inf_{\substack{U\subset H^1\\\dim U=k}}\,\max_{\substack{\|v\|_{L^2}=1\\v\in U}}\left(\int_{\R^N}|\nabla v|^2\,dx+\int_{\R^N}V_{\beta_n}(x)|v|^2\,dx\right)\\
\leq \max_{\substack{\|v\|_{L^2}=1\\v\in \span\langle \varphi_1,\dots,\varphi_k\rangle}}\left(\int_{\R^N}|\nabla v|^2\,dx+\int_{\R^N}V_{\beta_n}(x)|v|^2\,dx\right)\\
=\max_{\substack{\|v\|_{L^2}=1\\v\in \span\langle \varphi_1,\dots,\varphi_k\rangle}}\left(\int_{\R^N}|\nabla v|^2\,dx+\int_{\R^N}|v|^2\,dx\right)=\lambda_k<\sigma_{\beta_n,0}, \quad k=1,\dots,\bar k.
\end{multline} 
It follows from Theorem \ref{charact} that, for all $n\geq\bar n$, $\lambda_{\beta_n,1}$, \dots, $\lambda_{\beta_n \bar k}$ are eigenvalues of $A_{\beta_n}$ satisfying (\ref{proper-recursive}).

In order to prove statements (3) and (4) we procede by induction.  Let $\ell\in\{1,\dots,\bar k\}$ and assume that statements (3) and (4) are true for $k=1$, \dots, $\ell-1$. We shall prove that then statements  (3) and (4) are true for $k=1$, \dots, $\ell$. For ease of notation, set $\lambda_{\beta,0}:=0$ and $\lambda_{0}:=0$. 
We have
\begin{multline}
0\leq\lambda_{\beta_n,\ell-1}\leq\lambda_{\beta_n,\ell}=
\int_{\R^N}|\nabla \varphi_{\beta_n,\ell}|^2\,dx+\int_{\R^N}V_{\beta_n}(x)|\varphi_{\beta_n,\ell}|^2\,dx\\
=\|\varphi_{\beta_n,\ell}\|_{H^1(\R^N)}^2+\beta_n\int_{\R^N}V(x)|\varphi_{\beta_n,\ell}|^2\,dx\leq\lambda_\ell.
\end{multline}
It follows that there exists $\zeta\in[\lambda_{\ell-1},\lambda_\ell]$ such that, up to a subsequence, $\lambda_{\beta_n,\ell}\to\zeta$ as $n\to\infty$, and there exists $\tilde v\in H^1(\R^N)$ such that, up to a subsequence, $\varphi_{\beta_n,\ell}\weakto\tilde v$ in $H^1(\R^N)$  and $\varphi_{\beta_n,\ell}\to\tilde v$ in $L^2_{\loc}(\R^N)$ as $n\to\infty$. 

Let $\cl{B_r(x_0)}$ be a closed ball contained in $\R^N\setminus\cl\Omega$. Then we have
\begin{equation}\left(\essinf_{\cl{B_r(x_0)}} V(x)\right)\int_{\cl{B_r(x_0)}}|\varphi_{\beta_n,\ell}|^2\,dx\leq\frac{\lambda_\ell}{\beta_n}\to 0\quad\text{as $n\to\infty$.}\end{equation}
It follows that $\tilde v(x)=0$ a.e. in $\R^N\setminus\Omega$, i.e. $\tilde v\in H^1_\Omega(\Bbb R^N)$. 

Let $R>0$ be such that $\essinf_{\R^N\setminus\cl{B_R(0)}} V(x)\geq\frac12$. Then we have
\begin{equation}\frac12\int_{\R^N\setminus\cl{B_r(x_0)}}|\varphi_{\beta_n,\ell}|^2\,dx\leq\frac{\lambda_\ell}{\beta_n}\to 0\quad\text{as $n\to\infty$.}\end{equation}
This, together with the fact that $\varphi_{\beta_n,\ell}\to\tilde v$ in $L^2_{\loc}(\R^N)$, implies that $\varphi_{\beta_n,\ell}\to\tilde v$ in $L^2(\R^N)$ as $n\to\infty$. It follows that $\|\tilde v\|_{L^2_\Omega(\R^N)}=1$. Now, by the inductive hypothesis, for $k=1$, \dots, $\ell-1$ we have
\begin{equation}
0=\int_{\R^N}\varphi_{\beta_n,k}\varphi_{\beta_n,\ell}\,dx\to\int_{\R^N}\tilde \varphi_{k}\tilde v\,dx\quad\text{as $n\to\infty$,}
\end{equation}
so $\tilde v\perp\span\langle\tilde \varphi_1,\dots,\tilde \varphi_{\ell-1}\rangle$.
Now let $\phi\in H^1_\Omega(\R^N)$. For all $n\geq \bar n$ we have
\begin{multline}
\int_{\R^N}\nabla \varphi_{\beta_n,\ell}\cdot\nabla\phi\,dx+\int_{\R^N}\varphi_{\beta_n,\ell}\phi\,dx\\=\int_{\R^N}\nabla \varphi_{\beta_n,\ell}\cdot\nabla\phi\,dx+\int_{\R^N}V_{\beta_n}(x)\varphi_{\beta_n,\ell}\phi\,dx=\lambda_{\beta_n,\ell}\int_{\R^N}\varphi_{\beta_n,\ell}\phi\,dx.
\end{multline}
By letting $n\to\infty$ we get
\begin{equation}
\int_{\R^N}\nabla \tilde v\cdot\nabla\phi\,dx+\int_{\R^N}\tilde v\phi\,dx=\zeta\int_{\R^N}\tilde v\phi\,dx.
\end{equation}
Therefore $\zeta\in[\lambda_{\ell-1},\lambda_\ell]$ is an eigenvalue of $A_\Omega$, and $A_\Omega \tilde v=\zeta\tilde v$. If $\lambda_{\ell-1}=\lambda_\ell$, then $\zeta=\lambda_\ell$. If $\lambda_{\ell-1}<\lambda_\ell$, then either $\zeta=\lambda_{\ell-1}$ or $\zeta=\lambda_{\ell}$, but since $\tilde v\perp\span\langle\tilde \varphi_1,\dots,\tilde \varphi_{\ell-1}\rangle$ then necessarily $\zeta=\lambda_\ell$. We set $\tilde \varphi_\ell:=\tilde v$ and, in order to complete the proof, we only need to prove that $\|\varphi_{\beta_n,\ell}-\tilde \varphi_\ell\|_{1,\beta_n}\to 0$ as $n\to\infty$.
To this end, we observe that
\begin{multline*}
\|\varphi_{\beta_n,\ell}-\tilde\varphi_\ell\|_{1,\beta_n}^2\\
=\int_{\R^N}|\nabla(\varphi_{\beta_n,\ell}-\tilde\varphi_\ell)|^2\,dx+
\int_{\R^N}|\varphi_{\beta_n,\ell}-\tilde\varphi_\ell|^2\,dx+\beta_n\int_{\R^N}V(x)|\varphi_{\beta_n,\ell}-\tilde\varphi_\ell|^2\,dx\\
=\int_{\R^N}|\nabla \varphi_{\beta_n,\ell}|^2\,dx+
\int_{\R^N}|\varphi_{\beta_n,\ell}|^2\,dx+\beta_n\int_{\R^N}V(x)|\varphi_{\beta_n,\ell}|^2\,dx\\
-2\int_{R^N}\nabla \varphi_{\beta_n,\ell}\cdot\nabla \tilde\varphi_\ell\,dx-2\int_{\R^N}\varphi_{\beta_n,\ell}\tilde\varphi_\ell\,dx
+\int_{\R^N}|\nabla \tilde\varphi_\ell|^2\,dx+
\int_{\R^N}|\tilde\varphi_\ell|^2\,dx\\
=\lambda_{\beta_n,\ell}-2\int_{R^N}\nabla \varphi_{\beta_n,\ell}\cdot\nabla \tilde\varphi_\ell\,dx-2\int_{\R^N}\varphi_{\beta_n,\ell}\tilde\varphi_\ell\,dx+\lambda_\ell.
\end{multline*}
Since $\lambda_{\beta_n,\ell}\to\lambda_\ell$ and $\varphi_{\beta_n,\ell}\weakto\tilde\varphi_\ell$ in $H^1(\R^N)$ as $n\to\infty$, the proof is complete.\end{proof}

In the proof of Theorem \ref{convergenzaspettrale} the sequence  $(\varphi_{\beta_n,k})_n$ converges only up to a subsequence and the limit $\tilde\varphi_k$  depends on the particular subsequence of $(\beta_n)_{n\in\N}$. However, when $\lambda_{\cl k}<\lambda_{\cl k+1}$ the space generated by $\tilde\varphi_1$, \dots, $\tilde\varphi_{\cl k}$ is actually independent of the particular subsequence of $(\beta_n)_{n\in\N}$. Let us denote by $P_{\beta_n,\cl k}$ the spectral projection relative to the spectral set  $S_{\beta_n,\cl k}:=\{\lambda_{\beta_n,1},\dots,\lambda_{\beta_n,\cl k}\}$ of $A_{\beta_n}$, that is
\begin{equation}
P_{\beta_n,\cl k}u=\circantiorario_{\Gamma_{\beta_n,\cl k}} (\lambda I-A_{\beta_n})^{-1}u\,d\lambda=\sum_{k=1}^{\cl k}\langle u,\varphi_{\beta_n,k}\rangle_{L^2(\R^N)}\varphi_{\beta_n,k},\quad u\in H^1(\R^N),
\end{equation}
where $\Gamma_{\beta_n,\cl k}$ is a cycle in $\C\setminus \sigma(A_{\beta_n})$ with ${\rm Ind}_{\Gamma_{\beta_n,\cl k}}(S_{\beta_n,\cl k})=1$ and  ${\rm Ind}_{\Gamma_{\beta_n,\cl k}}(\sigma(A_{\beta_n})\setminus S_{\beta_n,\cl k})=0$. Moreover,
let us denote by $P_{\Omega,\cl k}$ the spectral projection relative to the spectral set  $S_{\Omega,\cl k}:=\{\lambda_{\Omega,1},\dots,\lambda_{\Omega,\cl k}\}$ of $A_{\Omega}$, that is
\begin{equation}
P_{\Omega,\cl k}u=\circantiorario_{\Gamma_{\Omega,\cl k}} (\lambda I-A_{\Omega})^{-1}u\,d\lambda=\sum_{k=1}^{\cl k}\langle u,\varphi_{\Omega,k}\rangle_{L^2_\Omega(\R^N)}\varphi_{\Omega,k},\quad u\in H^1_{\Omega},
\end{equation}
where $\Gamma_{\Omega, \cl k}$ is a cycle in $\C\setminus \sigma(A_{\Omega})$ with ${\rm Ind}_{\Gamma_{\Omega,\cl k}}(S_{\Omega,\cl k})=1$ and  ${\rm Ind}_{\Gamma_{\Omega,\cl k}}(\sigma(A_{\Omega})\setminus S_{\Omega,\cl k})=0$. Then one has
\begin{equation}
P_{\beta_n,\cl k}u_n\to P_{\Omega,\cl k} u\quad\text{as $n\to\infty$}
\end{equation}
whenever $(u_n)_n$ is a sequence in $L^2(\R^N)$ converging to some $u\in L^2_\Omega(\R^N)$.


\section{The linear hyperbolic equation and its semigroup}

For $\beta\geq0$ we denote by $X^1_\beta$ the Hilbert space $H^1(\R^N)\times L^2(\R^N)$ equipped with the scalar product
\begin{equation}
\langle(u,v),(h,k)\rangle_{X^1_\beta}:=\langle u,h\rangle_{1,\beta}+\langle v, k\rangle_{L^2(\R^N)}
\end{equation}
with the corresponding norm 
\begin{equation}\|(u,v)\|_{X^1_\beta}:=\left(\|u\|_{1,\beta}^2+\|v\|_{L^2(\R^N)}^2\right)^{1/2}.\end{equation} 
Moreover, we denote by $X^1_\Omega$ the Hilbert space $H^1_\Omega(\R^N)\times L^2_\Omega(\R^N)$ equipped with the scalar product
\begin{equation}
\langle(u,v),(h,k)\rangle_{X^1_\Omega}:=\langle u,h\rangle_{H^1_\Omega(\R^N)}+\langle v, k\rangle_{L^2_\Omega(\R^N)}
\end{equation}
with the corresponding norm 
\begin{equation}\|(u,v)\|_{X^1_\Omega}:=\left(\|u\|_{H^1_\Omega(\R^N)}^2+\|v\|_{L^2_\Omega(\R^N)}^2\right)^{1/2},\end{equation} 
and we notice that $X^1_\Omega$ is a closed subspace of $X^1_\beta$, and that the scalar product of $X^1_\beta$ computed on elements of $X^1_\Omega$ coincides with the scalar product of $X^1_\Omega$. 

Let $\gamma\geq0$. In the space $X^1_\beta$ we define the closed linear operator
\begin{equation}
B_{\beta}(u,v):=\left(v,-A_\beta u-\gamma v\right), \quad (u,v)\in D(B_\beta)= D(A_\beta)\times H^1(\R^N).
\end{equation}
Similarly, in the space $X^1_\Omega$ we define the closed linear operator
\begin{equation}
B_{\Omega}(u,v):=\left(v,-A_\Omega u-\gamma v\right), \quad (u,v)\in D(B_\Omega)=D(A_\Omega)\times H^1_\Omega(\R^N).
\end{equation}
\begin{Rem}We denote by the same symbols both the real operators defined above and their complexifications, acting in $H^1(\R^N,\C)\times L^2(\R^N,\C)$ and $H^1_\Omega(\R^N,\C)\times L^2(\R^N,\C)$ respectively.\end{Rem}

We shall prove the following theorem:
\begin{Theorem}\label{cond-HY} Let $\delta\geq0$ be such that 
\begin{equation}\label{condsudelta}
2\delta\leq\gamma\quad\text{and}\quad \gamma\delta<1.
\end{equation}
Then all $\mu>-\delta$ are in the resolvent set of $B_\beta$ for every $\beta\geq 0$, and in the resolvent set of $B_\Omega$. Moreover, there exist $M\geq 1$, dependent on $\delta$ but independent of $\beta$, such that for all $\mu>-\delta$ the following estimates hold:
\begin{equation}
\|(\mu I-B_\beta)^{-j}\|_{\mathcal L(X^1_\beta)}\leq \frac{M}{(\mu+\delta)^j}, \quad j=1,2,3,\dots
\end{equation}
and
\begin{equation}
\|(\mu I-B_\Omega)^{-j}\|_{\mathcal L(X^1_\Omega)}\leq \frac{M}{(\mu+\delta)^j}, \quad j=1,2,3,\dots
\end{equation}
If $\delta=0$ one can take $M=1$.
\end{Theorem}

In order to prove Theorem \ref{cond-HY} we need to introduce new scalar products in $X^1_\beta$ and $X^1_\Omega$. For $\delta$ satisfying (\ref{condsudelta}) we define
\begin{equation}
\langle(u,v),(h,k)\rangle_{X^1_\beta,\delta}:=\langle u,h\rangle_{1,\beta}+(\delta^2-\gamma\delta)\langle u, h\rangle_{L^2(\R^N)}+\langle v+\delta u, k+\delta h\rangle_{L^2(\R^N)}
\end{equation}
and
\begin{equation}
\langle(u,v),(h,k)\rangle_{X^1_\Omega,\delta}:=\langle u,h\rangle_{H^1_\Omega(\R^N)}+(\delta^2-\gamma\delta)\langle u, h\rangle_{L^2_\Omega(\R^N)}+\langle v+\delta u, k+\delta h\rangle_{L^2_\Omega(\R^N)},
\end{equation}
and we denote by $\|\cdot\|_{X^1_\beta,\delta}$ and $\|\cdot\|_{X^1_\Omega,\delta}$ the corresponding norms.

\begin{Lemma}\label{equivalent}
Let $\delta$ satisfy (\ref{condsudelta}). There exists two constants $k_\delta$ and $K_\delta$ such that 
\begin{equation}
k_\delta\|(u,v)\|_{X^1_\beta}\leq\|(u,v)\|_{X^1_\beta,\delta}\leq K_\delta\|(u,v)\|_{X^1_\beta},\quad\text{$(u,v)\in X^1_\beta$}
\end{equation}
for all $\beta\geq0$, and
\begin{equation}
k_\delta\|(u,v)\|_{X^1_\Omega}\leq\|(u,v)\|_{X^1_\Omega,\delta}\leq K_\delta\|(u,v)\|_{X^1_\Omega},\quad\text{$(u,v)\in X^1_\Omega$.}
\end{equation}
If $\delta=0$ one can take $k_\delta=K_\delta=1$.
\end{Lemma}
\begin{proof}
First we observe that since $2\delta\leq\gamma$, then $\delta^2-\gamma\delta\leq-\delta^2$. It follows that
\begin{multline}
\|(u,v)\|_{X^1_\beta,\delta}^2=\|u\|_{1,\beta}^2+(\delta^2-\gamma\delta)\|u\|_{L^2(\R^N)}^2+\|v+\delta u\|_{L^2(\R^N)}^2\\
=\|u\|_{1,\beta}^2+(\delta^2-\gamma\delta)\|u\|_{L^2(\R^N)}^2+\|v\|_{L^2(\R^N)}^2+\delta^2\| u\|_{L^2(\R^N)}^2+2\Re\langle v,\delta u \rangle_{L^2(\R^N)}\\
\leq \|u\|_{1,\beta}^2+(\delta^2-\gamma\delta)\|u\|_{L^2(\R^N)}^2+\|v\|_{L^2(\R^N)}^2+\delta^2\| u\|_{L^2(\R^N)}^2+\|v\|_{L^2(\R^N)}^2+\delta^2\|u\|_{L^2(\R^N)}^2\\
\leq  \|u\|_{1,\beta}^2+\delta^2\|u\|_{L^2(\R^N)}^2+2\|v\|_{L^2(\R^N)}^2\leq(1+\delta^2)\|u\|_{1,\beta}^2+2\|v\|_{L^2(\R^N)}^2\\
\leq(2+\delta^2)\|(u,v)\|_{X^1_\beta}^2
\end{multline}
On the other hand, we have
\begin{multline}
\|(u,v)\|_{X^1_\beta,\delta}^2=\|u\|_{1,\beta}^2+(\delta^2-\gamma\delta)\|u\|_{L^2(\R^N)}^2+\|v+\delta u\|_{L^2(\R^N)}^2\\
=\|u\|_{1,\beta}^2+(\delta^2-\gamma\delta)\|u\|_{L^2(\R^N)}^2+\|v\|_{L^2(\R^N)}^2+\delta^2\| u\|_{L^2(\R^N)}^2+2\Re\langle v,\delta u \rangle_{L^2(\R^N)}\\
\geq \|u\|_{1,\beta}^2+(\delta^2-\gamma\delta)\|u\|_{L^2(\R^N)}^2+\|v\|_{L^2(\R^N)}^2+\delta^2\| u\|_{L^2(\R^N)}^2-\frac12\|v\|_{L^2(\R^N)}^2-2\delta^2\|u\|_{L^2(\R^N)}^2\\
= \|u\|_{1,\beta}^2-\gamma\delta\|u\|_{L^2(\R^N)}^2+\frac12\|v\|_{L^2(\R^N)}^2\geq (1-\gamma\delta) \|u\|_{1,\beta}^2+\frac12\|v\|_{L^2(\R^N)}^2\\
\geq\min\{1-\gamma\delta,1/2\}\|(u,v)\|_{X^1_\beta}^2
\end{multline}
The proof for $\|(u,v)\|_{X^1_\Omega,\delta}$ is completely analogous.
\end{proof}

\begin{Lemma}\label{dissip} Let $\delta$ satisfy (\ref{condsudelta}). Then 
\begin{equation}
\Re \langle (B_\beta+\delta I)(u,v),(u,v)\rangle_{X^1_\beta,\delta}= -(\gamma-2\delta)\|v+\delta u\|_{L^2(\R^N)}^2,\quad\text{$(u,v)\in D(B_\beta)$}
\end{equation}
for all $\beta\geq 0$ and
\begin{equation} \Re\langle (B_\Omega+\delta I)(u,v),(u,v)\rangle_{X^1_\Omega,\delta}= -(\gamma-2\delta)\|v+\delta u\|_{L^2_\Omega(\R^N)}^2,\quad\text{$(u,v)\in D(B_\Omega)$.}\end{equation}
\end{Lemma}
\begin{proof} We have
\begin{multline}
\langle (B_\beta+\delta I)(u,v),(u,v)\rangle_{X^1_\beta,\delta}=\langle (v,-A_\beta u-\gamma v),(u,v)\rangle_{X^1_\beta,\delta}+\delta\langle (u,v),(u,v)\rangle_{X^1_\beta,\delta}\\
=\langle v,u\rangle_{1,\beta}+(\delta^2-\gamma\delta)\langle v,u\rangle_{L^2(\R^N)}+\langle -A_\beta u-\gamma v+\delta v,v+\delta u\rangle_{L^2(\R^N)}\\
+\delta\langle u,u\rangle_{1,\beta}+\delta(\delta^2-\gamma\delta)\langle u, u\rangle_{L^2(\R^N)}+\delta\langle v+\delta u, v+\delta u\rangle_{L^2(\R^N)}\\
=\langle v,u\rangle_{1,\beta}+(\delta^2-\gamma\delta)\langle v,u\rangle_{L^2(\R^N)}-\langle u,v\rangle_{1,\beta}-\delta\langle u,u\rangle_{1,\beta}-(\gamma-\delta)\langle v,v\rangle_{L^2(\R^N)}\\
-\delta(\gamma-\delta)\langle v,u\rangle_{L^2(\R^N)}+\delta\langle u,u\rangle_{1,\beta}+\delta(\delta^2-\gamma\delta)\langle u, u\rangle_{L^2(\R^N)}+\delta\langle v+\delta u, v+\delta u\rangle_{L^2(\R^N)}\\
=2{\rm i}\Im\langle v,u\rangle_{1,\beta}+2(\delta^2-\gamma\delta)\langle v,u\rangle_{L^2(\R^N)} -(\gamma-\delta)\langle v,v\rangle_{L^2(\R^N)}\\
+\delta(\delta^2-\gamma\delta)\langle u, u\rangle_{L^2(\R^N)}+\delta\langle v+\delta u, v+\delta u\rangle_{L^2(\R^N)}\\
=2{\rm i}\left(\Im\langle v,u\rangle_{1,\beta}+\Im(\delta^2-\gamma\delta)\langle v,u\rangle_{L^2(\R^N)}\right)\\
-(\gamma-\delta)\langle v+\delta u, v+\delta u\rangle_{L^2(\R^N)}+\delta\langle v+\delta u, v+\delta u\rangle_{L^2(\R^N)}\\
=2{\rm i}\left(\Im\langle v,u\rangle_{1,\beta}+\Im(\delta^2-\gamma\delta)\langle v,u\rangle_{L^2(\R^N)}\right)-(\gamma-2\delta)\langle v+\delta u, v+\delta u\rangle_{L^2(\R^N)}.
\end{multline}
The proof for $\langle (B_\Omega+\delta I)(u,v),(u,v)\rangle_{X^1_\Omega,\delta}$ is completely analogous.
\end{proof}

\begin{proof}[Proof of Theorem \ref{cond-HY}]
For $\mu\in\C$ and $(h,k)\in X^1_\beta$ let us consider the problem
$(B_\beta-\mu I)(u,v)=(h,k)$, that is
\begin{equation}
\begin{cases}v-\mu u=h\\-A_\beta u-\gamma v-\mu v=k\end{cases}
\end{equation}
By the first equation we get $v=\mu u+h$. Substituting this expression in the second equation we obtain the following equation for $u$:
\begin{equation}
-A_\beta u-\mu(\gamma+\mu)u=k+(\gamma+\mu)h.
\end{equation}
It follows that $\mu$ is in the resolvent set of $B_\beta$ if and only if $-\mu(\gamma+\mu)$ is in the resolvent set of $A_\beta$. If this is the case, we can easily compute the resolvent operator:
\begin{equation}\label{resolvent_beta}(B_\epsilon-\mu I)^{-1}\left(\begin{array}{l}h\\k\end{array}\right)=\left(\begin{array}{l}(-A_\beta-\mu(\gamma+\mu)I)^{-1}(k+(\gamma+\mu)h)\\h+
\mu(-A_\beta-\mu(\gamma+\mu)I)^{-1}(k+(\gamma+\mu)h)
\end{array}\right).\end{equation}
Analogously,  $\mu$ is in the resolvent set of $B_\Omega$ if and only if $-\mu(\gamma+\mu)$ is in the resolvent set of $A_\Omega$ and if this is the case the resolvent operator is
\begin{equation}\label{resolvent_Omega}(B_\Omega-\mu I)^{-1}\left(\begin{array}{l}h\\k\end{array}\right)=\left(\begin{array}{l}(-A_\Omega-\mu(\gamma+\mu)I)^{-1}(k+(\gamma+\mu)h)\\h+
\mu(-A_\Omega-\mu(\gamma+\mu)I)^{-1}(k+(\gamma+\mu)h)
\end{array}\right).\end{equation}
If $\mu>-\delta$ then $-\mu(\gamma+\mu)<\delta\gamma<1$, and therefore $-\mu(\gamma+\mu)$ belongs to the resolvent set of $A_\beta$ and $A_\Omega$. It follows that $\mu$ belongs to the resolvent set of $B_\beta$ and $B_\Omega$.

Now let $\zeta>0$ and $U=(u,v)\in D(B_\beta)$. By Lemma \ref{dissip} we have that
\begin{multline}\|(I-\zeta (B_\beta+\delta I))U\|_{X^1_\beta,\delta}^2=\langle U-\zeta (B_\beta +\delta I)U,U-\zeta (B_\beta+\delta I) U\rangle_{X^1_\beta,\delta}\\
=\|U\|_{X^1_\beta,\delta}^2-2\zeta\Re\langle (B_\beta +\delta I)U,U\rangle_{X^1_\beta,\delta}+\zeta^2\| (B_\beta+\delta I)U\|_{X^1_\beta,\delta}^2\geq \|U\|_{X^1_\beta,\delta}^2\end{multline}
It follows that for all $\zeta>0$
\begin{equation}\|((\zeta^{-1}-\delta)I-B_\beta)U\|_{X^1_\beta,\delta}\geq\zeta^{-1} \|U\|_{X^1_\beta,\delta}.\end{equation}
Setting $\mu:=\zeta^{-1}-\delta$ and taking $U=(\mu I-B_\beta)^{-1} \tilde U$, $\tilde U\in X^1_\beta$, we get that  
\begin{equation}
\|((\mu I-B_\beta)^{-1}\tilde U\|_{X^1_\beta,\delta}\leq \frac1{\mu+\delta} \|\tilde U\|_{X^1_\beta,\delta}
\end{equation}
for all $\mu>-\delta$. Iterating this inequality, we get
\begin{equation}
\|(\mu I-B_\beta)^{-j}\tilde U\|_{X^1_\beta,\delta}\leq \frac{1}{(\mu+\delta)^j}\|\tilde U\|_{X^1_\beta,\delta}, \quad j=1,2,3,\dots
\end{equation}
The conclusion now follows from Lemma \ref{equivalent}. By the same arguments we obtain the thesis for the operator $B_\Omega$.
\end{proof}

\begin{Theorem}\label{semigroup}
Let $\delta$ and $M$ be as in Theorem \ref{cond-HY}. For all $\beta\geq 0$ the operator $B_\beta$ is the infinitesimal generator of a strongly continuous semigroup $T_\beta(t)$, $t\geq 0$, satisfying the estimate
\begin{equation}
\|T_\beta(t)U\|_{X^1_\beta}\leq Me^{-\delta t}\|U\|_{X^1_\beta}\quad t\geq 0,\, U\in X^1_\beta.
\end{equation}
Moreover, the operator $B_\Omega$ is the infinitesimal generator of a strongly continuous semigroup $T_\Omega(t)$, $t\geq 0$, satisfying the estimate
\begin{equation}
\|T_\Omega(t)U\|_{X^1_\Omega}\leq Me^{-\delta t}\|U\|_{X^1_\Omega}\quad t\geq 0,\, U\in X^1_\Omega.
\end{equation}
If $\delta=0$ we can take $M=1$.
\end{Theorem}
\begin{proof}Theorem \ref{cond-HY} guarantees that $B_\beta$ and $B_\Omega$ satisfy the condition of the Hille-Yosida-Feller-Miyadera-Phillips Theorem (see e.g. \cite[Theorem 2.9.1]{Ves} and \cite{Pazy}). The conclusion follows.\end{proof}

The semigroups $T_\beta(t)$ and $T_\Omega(t)$ provide solutions to 
\begin{equation}
\begin{cases}u_t=v\\
v_t=-A_\beta u-\gamma v\\
u(0)=u_0,\quad v(0)=v_0
\end{cases}
\end{equation}
and 
\begin{equation}
\begin{cases}u_t=v\\
v_t=-A_\Omega u-\gamma v\\
u(0)=u_0,\quad v(0)=v_0
\end{cases}
\end{equation}
in the following sense: if $U_0=(u_0,v_0)\in D(B_\beta)$ (resp. $U_0=(u_0,v_0)\in D(B_\Omega)$), then $U(t):=T_\beta(t)U_0$ (resp. $U(t):=T_\Omega(t)U_0$) is continuous into $D(B_\beta)$ (resp. into $D(B_\Omega)$), differentiable into $X^1_\beta$ (resp. into $X^1_\Omega$), and $\dot U(t)=B_\beta U(t)$ (resp. $\dot U(t)=B_\Omega U(t)$). If $U_0\in X^1_\beta$ (resp. $U_0\in X^1_\Omega$) this is no longer true. However, $U(t)$ is differentiable into $X^0_\beta:=L^2(\R^N)\times H^{-1}(\R^N)$ (resp. into $X^0_\Omega:=L^2_\Omega(\R^N)\times H^{-1}_\Omega(\R^N)$), and $\dot U(t)=\Theta_\beta U(t)$ (resp. $\dot U(t)=\Theta_\Omega U(t)$), where
$$
\Theta_\beta (u,v)=(v, -\Lambda_\beta u-\gamma v), \quad (u,v)\in D(\Theta_\beta)= X^1_\beta
$$
and 
$$
\Theta_\Omega (u,v)=(v, -\Lambda_\Omega u-\gamma v), \quad (u,v)\in D(\Theta_\Omega)= X^1_\Omega.
$$
This means that if $U(t)=(u(t),v(t))$ is defined as above, then the equations
$$u_{tt}+\gamma u_t+V_\beta u -\Delta u=0, \quad (t,x)\in [0,+\infty[\times\R^n$$
and 
$$\begin{aligned}u_{tt}+\gamma u_t+u-\Delta u&=0, && (t,x)\in [0,+\infty[\times\Omega\\u&=0,&&(t,x)\in[0,+\infty[\times\partial\Omega\end{aligned}$$ 
are satisfied in $H^{-1}(\R^N)$ and $H^{-1}(\Omega)$ respectively.

If $h\colon[0,T]\to X^1_\beta$ (resp. $h\colon[0,T]\to X^1_\Omega$) is continuous and $U_0\in X^1_\beta$ (resp. $U_0\in X^1_\Omega$), then
\begin{equation}
U(t):=T_\beta(t)U_0+\int_0^t T_\beta(t-s)h(s)\,ds
\end{equation}
and
\begin{equation}
U(t):=T_\Omega(t)U_0+\int_0^t T_\Omega(t-s)h(s)\,ds
\end{equation}
are called the {\em mild solutions} of the inhomogeneous equations
\begin{equation}
\dot U=B_\beta U+h(t),\quad U(0)=U_0
\end{equation}
and 
\begin{equation}
\dot U=B_\Omega U+h(t),\quad U(0)=U_0
\end{equation}
respectively. The inhomogeneous equations are satisfied in the following sense: $U(t)$ is differentiable into $X^0_\beta$ (resp. into $X^0_\Omega$), and $\dot U(t)=\Theta_\beta U(t)+h(t)$ (resp. $\dot U(t)=\Theta_\Omega U(t)+h(t)$). These are well known facts, see e.g. \cite{CH} and \cite{Pazy}; a short recap with self-contained proofs which fits with the notations of the present  paper is contained in \cite{PR}.


\section{A singular Trotter-Kato theorem} 

In this section we state and prove a singular version of the Trotter-Kato theorem for the semigroups $T_\beta(t)$ as $\beta\to+\infty$. This result differs from the standard Trotter-Kato theorem  since we do not work with a fixed norm in $H^1(\R^N)$ but with a family of norms that diverge as $\beta\to+\infty$. We begin with a convergence result for the resolvents of the operators $B_\beta$.

\begin{Theorem}\label{resolvconverg2} 
Let $(\beta_n)_{n\in\N}$ be a sequence of positive numbers, $\beta_n\to+\infty$ as $n\to\infty$. Let $H=(h,k)\in X^1_\Omega$ and for every $n\in\N$ let $H_n=(h_n,k_n)\in X^1_{\beta_n}$, and assume that $\|H_n-H\|_{X^1_{\beta_n}}\to 0$ as $n\to\infty$. Let $\mu>-\delta$ be fixed. Then
\begin{equation}\label{x1}
\|(B_{\beta_n}-\mu I)^{-1}H_n-(B_{\Omega}-\mu I)^{-1}H\|_{X^1_{\beta_n}}\to 0\quad\text{as $n\to\infty$.}\end{equation}
Moreover,
\begin{equation}\label{x2}
\|\left((B_{\beta_n}-\mu I)^{-1}-(B_{\Omega}-\mu I)^{-1}\right)W\|_{X^1_{\beta_n}}\to 0\quad\text{as $n\to\infty$.}\end{equation}
uniformly with respect to $W$ in any compact subset of $X^1_\Omega$.
\end{Theorem}
\begin{proof}
The first statement is a straightforward consequence of Corollary \ref{resolvconverg} and of formulas (\ref{resolvent_beta}) and (\ref{resolvent_Omega}). 
Now let $\mathcal K\subset X^1_\Omega$ be compact. We want to prove that
\begin{equation}
\sup_{W\in\mathcal K}\|\left((B_{\beta_n}-\mu I)^{-1}-(B_{\Omega}-\mu I)^{-1}\right)W\|_{X^1_{\beta_n}}\to 0\quad\text{as $n\to\infty$.}
\end{equation}
Assume by contradiction that this is not true. Then there exists $\epsilon>0$ and for all $k\in\N$ there exist $n_k>\max\{k,n_{k-1}\}$ and $W_k\in\mathcal K$ such that
\begin{equation}
\|\left((B_{\beta_{n_k}}-\mu I)^{-1}-(B_{\Omega}-\mu I)^{-1}\right)W_k\|_{X^1_{\beta_{n_k}}}\geq\epsilon.
\end{equation}
Since $\mathcal K$ is compact, there exists $\cl W\in\mathcal K$ such that, up to a subsequence,
$W_k\to \cl W$ in $X^1_\Omega$ as $k\to\infty$. Then by (\ref{x1})
\begin{equation}
\|(B_{\beta_{n_k}}-\mu I)^{-1}W_k-(B_{\Omega}-\mu I)^{-1}\cl W\|_{X^1_{\beta_{n_k}}}\to 0\quad\text{as $k\to\infty$.}
\end{equation}
Therefore we have
\begin{multline}
\|\left((B_{\beta_{n_k}}-\mu I)^{-1}-(B_{\Omega}-\mu I)^{-1}\right)W_k\|_{X^1_{\beta_{n_k}}}\\
\leq \|(B_{\beta_{n_k}}-\mu I)^{-1}W_k-(B_{\Omega}-\mu I)^{-1}\cl W\|_{X^1_{\beta_{n_k}}}+\|(B_\Omega-\mu I)^{-1}\cl W-(B_{\Omega}-\mu I)^{-1}W_k\|_{X^1_{\beta_{n_k}}}\\
= \|(B_{\beta_{n_k}}-\mu I)^{-1}W_k-(B_{\Omega}-\mu I)^{-1}\cl W\|_{X^1_{\beta_{n_k}}}+\|(B_\Omega-\mu I)^{-1}\cl W-(B_{\Omega}-\mu I)^{-1}W_k\|_{X^1_\Omega}\\
\to 0\quad\text{as $k\to\infty$,}
\end{multline} a contradiction. The proof is complete.
\end{proof}

We can now state and prove the singular version of the Trotter-Kato theorem. We follow closely the proof of the standard version, for which we refer e.g. to \cite[Th. 3.17]{Davies} or \cite[Th. 5.3.1]{Ves}.

\begin{Theorem}\label{TrKa}
Let $(\beta_n)_{n\in\N}$ be a sequence of positive numbers, $\beta_n\to+\infty$ as $n\to\infty$. Let $U_n=(u_n,v_n)\in X^1_{\beta_n}$ for $n\in\N$, let $U=(u,v)\in X^1_\Omega$,  and assume that $\|U_n -U\|_{X^1_{\beta_n}}\to 0$ as $n\to\infty$. Then for all $\tau>0$
\begin{equation}\label{x3}
\sup_{t\in[0,\tau]}\|T_{\beta_n}(t)U_n-T_{\Omega}(t)U\|_{X^1_{\beta_n}}\to 0\quad\text{as $n\to\infty$}.
\end{equation}
\end{Theorem}
\begin{proof}
First we observe that for $t\in[0,\tau]$
\begin{multline}
\|T_{\beta_n}(t)U_n-T_{\Omega}(t)U\|_{X^1_{\beta_n}}\\
\leq \|T_{\beta_n}(t)U_n-T_{\beta_n}(t)U\|_{X^1_{\beta_n}}+
\|T_{\beta_n}(t)U-T_{\Omega}(t)U\|_{X^1_{\beta_n}}\\
\leq Me^{-\delta t}\|U_n-U\|_{X^1_{\beta_n}}+\|T_{\beta_n}(t)U-T_{\Omega}(t)U\|_{X^1_{\beta_n}}\\
\leq M\|U_n-U\|_{X^1_{\beta_n}}+\|T_{\beta_n}(t)U-T_{\Omega}(t)U\|_{X^1_{\beta_n}},
\end{multline}
where $M$ and $\delta$ are the constants of Theorem \ref{semigroup}. Therefore it is sufficient to prove that if $U\in X^1_\Omega$ then for all $\tau>0$
\begin{equation}
\sup_{t\in[0,\tau]}\|T_{\beta_n}(t)U-T_{\Omega}(t)U\|_{X^1_{\beta_n}}\to 0\quad\text{as $n\to\infty$}.
\end{equation}
Now we fix $\mu>0$ and we observe that, for $U\in X^1_\Omega$ and for $0\leq s\leq t$ we have
\begin{multline}
\frac{d}{ds}\left( T_{\beta_n}(t-s)(\mu I-B_{\beta_n})^{-1}T_{\Omega}(s)(\mu I-B_{\Omega})^{-1} U\right)\\
=-B_{\beta_n}T_{\beta_n}(t-s) (\mu I-B_{\beta_n})^{-1}T_{\Omega}(s)(\mu I-B_{\Omega})^{-1} U\\+T_{\beta_n}(t-s) (\mu I-B_{\beta_n})^{-1}B_\Omega T_{\Omega}(s)(\mu I-B_{\Omega})^{-1} U\\
=T_{\beta_n}(t-s)(-B_{\beta_n} )(\mu I-B_{\beta_n})^{-1}T_{\Omega}(s)(\mu I-B_{\Omega})^{-1} U\\-T_{\beta_n}(t-s) (\mu I-B_{\beta_n})^{-1} T_{\Omega}(s)(-B_\Omega)(\mu I-B_{\Omega})^{-1} U\\
=T_{\beta_n}(t-s)\left((\mu I-B_{\Omega})^{-1}-(\mu I-B_{\beta_n})^{-1}\right)T_{\Omega}(s) U.
\end{multline}
In follows that
\begin{multline}
(\mu I-B_{\beta_n})^{-1}T_{\Omega}(t)(\mu I-B_{\Omega})^{-1} U
- T_{\beta_n}(t)(\mu I-B_{\beta_n})^{-1}(\mu I-B_{\Omega})^{-1} U\\
=\int_0^t T_{\beta_n}(t-s)\left((\mu I-B_{\Omega})^{-1}-(\mu I-B_{\beta_n})^{-1}\right)T_{\Omega}(s) U\,ds,\end{multline}
whence
\begin{multline}
(\mu I-B_{\beta_n})^{-1}\left(T_{\Omega}(t)- T_{\beta_n}(t)\right)(\mu I-B_{\Omega})^{-1} U\\
=\int_0^t T_{\beta_n}(t-s)\left((\mu I-B_{\Omega})^{-1}-(\mu I-B_{\beta_n})^{-1}\right)T_{\Omega}(s) U\,ds.\end{multline}
Since the set $\{T_\Omega(s)U\mid s\in[0,\tau]\}$ is compact in $H^1_\Omega(\R^N)$, it follows from Theorem \ref{resolvconverg2} that for all $t\in[0,\tau]$
\begin{multline}
\|(\mu I-B_{\beta_n})^{-1}\left(T_{\Omega}(t)- T_{\beta_n}(t)\right)(\mu I-B_{\Omega})^{-1} U\|_{X^1_{\beta_n}}\\
\leq M \int_0^\tau\|\left((\mu I-B_{\Omega})^{-1}-(\mu I-B_{\beta_n})^{-1}\right)T_{\Omega}(s) U\|_{X^1_{\beta_n}}\,ds\to 0\quad\text{as $n\to\infty$.}
\end{multline}
Therefore we have that for all $U\in D(B_\Omega)$
\begin{equation}\label{provvisorio}
\sup_{t\in[0,\tau]}\|(\mu I-B_{\beta_n})^{-1}\left(T_{\Omega}(t)- T_{\beta_n}(t)\right)U\|_{X^1_{\beta_n}}\to 0 \quad\text{as $n\to\infty$,}
\end{equation}
and by a standard density argument the same is true for all $U\in X^1_\Omega$.
Now we observe that
\begin{multline}
(\mu I-B_{\beta_n})^{-1}\left(T_{\Omega}(t)- T_{\beta_n}(t)\right)U
\\
=\left((\mu I-B_{\beta_n})^{-1}-(\mu I-B_{\beta_\Omega})^{-1}\right)T_{\Omega}(t)U\\+\left(T_{\Omega}(t)-T_{\beta_n(t)}\right)(\mu I-B_{\beta_\Omega})^{-1}U\\+T_{\beta_n}(t)\left((\mu I-B_{\Omega})^{-1}-(\mu I-B_{\beta_n})^{-1}\right)U.
\end{multline}
Now it follows from (\ref{provvisorio})  and from Theorem \ref{resolvconverg2} that 
for all $U\in D(B_\Omega)$
\begin{equation}
\sup_{t\in[0,\tau]}\|\left(T_{\Omega}(t)- T_{\beta_n}(t)\right)U\|_{X^1_{\beta_n}}\to 0 \quad\text{as $n\to\infty$,}
\end{equation}
and again by a standard density argument the same is true for all $U\in X^1_\Omega$. The proof is complete.
\end{proof}

We state and demonstrate some consequences of Theorem \ref{TrKa} which will be needed later, in dealing with the nonlinear problem.

\begin{Corol}\label{uniformTrKa}
Let $(\beta_n)_{n\in\N}$ be a sequence of positive numbers, $\beta_n\to+\infty$ as $n\to\infty$. Let $\mathcal K$ be a compact subset of $X^1_\Omega$. Then for all $\tau>0$
\begin{equation}
\sup_{t\in[0,\tau]}\|(T_{\beta_n}(t)-T_{\Omega}(t))W\|_{X^1_{\beta_n}}\to 0\quad\text{as $n\to\infty$}
\end{equation}
uniformly with respect to $W$ in $\mathcal K$.
\end{Corol}
\begin{proof}
We want to prove that
\begin{equation}
\sup_{W\in\mathcal K}\sup_{t\in[0,\tau]}\|(T_{\beta_n}(t)-T_{\Omega}(t))W\|_{X^1_{\beta_n}}\to 0\quad\text{as $n\to\infty$.}
\end{equation}
Assume by contradiction that this is not true. Then there exists $\epsilon>0$ and for all $k\in\N$ there exist $n_k>\max\{k,n_{k-1}\}$, $t_k\in[0,\tau]$ and $W_k\in\mathcal K$ such that
\begin{equation}
\|(T_{\beta_{n_k}}(t_k)-T_{\Omega}(t_k))W_k\|_{X^1_{\beta_{n_k}}}\geq\epsilon.
\end{equation}
Since $\mathcal K$ is compact, there exists $\cl W\in\mathcal K$ such that, up to a subsequence,
$W_k\to \cl W$ in $X^1_\Omega$ as $k\to\infty$. Then by (\ref{x3})
\begin{equation}
\|T_{\beta_{n_k}}(t_k)W_k-T_{\Omega}(t_k)\cl W\|_{X^1_{\beta_{n_k}}}\to 0\quad\text{as $k\to\infty$.}
\end{equation}
Therefore we have
\begin{multline}
\|(T_{\beta_{n_k}}(t_k)-T_{\Omega}(t_k))W_k\|_{X^1_{\beta_{n_k}}}\\
\leq \|T_{\beta_{n_k}}(t_k)W_k-T_\Omega(t_k))\cl W\|_{X^1_{\beta_{n_k}}}+\|T_{\Omega}(t_k)(\cl W-W_k)\|_{X^1_{\beta_{n_k}}}\\
\leq \|T_{\beta_{n_k}}(t_k)W_k-T_\Omega(t_k))\cl W\|_{X^1_{\beta_{n_k}}}+\|T_{\Omega}(t_k)(\cl W-W_k)\|_{X^1_\Omega}\\
\leq \|T_{\beta_{n_k}}(t_k)W_k-T_\Omega(t_k))\cl W\|_{X^1_{\beta_{n_k}}}+M\|\cl W-W_k\|_{X^1_\Omega}
\to 0\quad\text{as $k\to\infty$,}
\end{multline} a contradiction. The proof is complete.
\end{proof}

We need the following Lemma.

\begin{Lemma}\label{uniformstrongcont}
Let $\mathcal K$ be a compact subset of $X^1_\Omega$. Then for every $\epsilon>0$ there exists $\tau>0$ such that 
\begin{equation}
\sup_{t\in[0,\tau]}\|T_\Omega(t)\cl U-\cl U\|_{X^1_\Omega}<\epsilon
\end{equation}
uniformly with respect to $\cl U\in{\mathcal K}$.
\end{Lemma}
\begin{proof}
Assume by contradiction that this is not true. Then there exists $\epsilon>0$ and for every $k\in\N$ there exist $t_k\in[0,1/k]$ and $\cl U_k\in\mathcal K$ such that
\begin{equation*}
\|T_\Omega(t_k)\cl U_k-\cl U_k\|_{X^1_\Omega}\geq\epsilon.
\end{equation*}
Since $\mathcal K$ is compact, there exists $\cl U\in\mathcal K$ such that, up to a subsequence, $\cl U_k\to\cl U$ in $X^1_\Omega$ as $k\to\infty$. It follows that
\begin{multline*}
\|T_\Omega(t_k)\cl U-\cl U\|_{X^1_\Omega}=\|T_\Omega(t_k)\cl U-T_\Omega(t_k)\cl U_k+T_\Omega(t_k)\cl U_k-\cl U_k+\cl U_k-\cl U\|_{X^1_\Omega}\\
\geq \|T_\Omega(t_k)\cl U_k-\cl U_k\|_{X^1_\Omega}- \|T_\Omega(t_k)\cl U-T_\Omega(t_k)\cl U_k\|_{X^1_\Omega}-\|\cl U_k-\cl U\|_{X^1_\Omega}\\
\geq  \epsilon- (M+1)\|\cl U_k-\cl U\|_{X^1_\Omega}.
\end{multline*}
Letting $k\to\infty$ we get $0\geq\epsilon$, a contradiction.
\end{proof}
 
 Now we can state and demonstrate the second corollary.
 
\begin{Corol}\label{corolxyz}
Let  $\mathcal K$ be a compact subset of $X^1_\Omega$  and let $D>0$. There exist $\rho>0$, $\tau>0$ and $\cl\beta>0$ such that, whenever $\cl U\in \mathcal K$,
\begin{enumerate}
\item for all $U\in X^1_\Omega$ with $\|U-\cl U\|_{X^1_\Omega}\leq\rho$ and for all $t\in[0,\tau]$
\begin{equation}
\|T_\Omega(t) U-U\|_{X^1_\Omega}\leq D
\end{equation}
\item for all $\beta\geq\cl \beta$, for all $U\in X^1_\beta$ with $\|U-\cl U\|_{X^1_\beta}\leq\rho$ and for all $t\in[0,\tau]$
\begin{equation}
\|T_\beta(t) U-U\|_{X^1_\beta}\leq D.
\end{equation}
\end{enumerate}
\end{Corol}
\begin{proof}
Let $M>0$ be as in Theorem \ref{semigroup}. Let $\rho$ be such that $3(M+1)\rho\leq D$. By Lemma \ref{uniformstrongcont}  there exists $\tau>0$ such that for all $\cl U\in\mathcal K$
\begin{equation}
\sup_{t\in[0,\tau]}\|T_\Omega(t)\cl U-\cl U\|_{X^1_\Omega}
\leq D/3.
\end{equation}
By Corollary \ref{uniformTrKa} there exists $\cl\beta>0$ such that for all $\cl U\in\mathcal K$ and all $\beta\geq\cl\beta$ 
\begin{equation}
\sup_{t\in[0,\tau]}\|T_\beta(t)\cl U-T_\Omega(t)\cl U\|_{X^1_\beta}\leq D/3.
\end{equation}
Now fix $\cl U\in\mathcal K$. For $U\in X^1_\Omega$ with $\|U-\cl U\|_{X^1_\Omega}\leq\rho$ and for $t\in[0,\tau]$
we have
\begin{multline}
\|T_\Omega(t) U- U\|_{X^1_\Omega}\\
\leq\|T_\Omega(t)U-T_\Omega(t)\cl U\|_{X^1_\Omega}+\|T_\Omega(t)\cl U-\cl U\|_{X^1_\Omega}+\|\cl U- U\|_{X^1_\Omega}\\
\leq M\|U-\cl U\|_{X^1_\Omega}+\sup_{t\in[0,\tau]}\|T_\Omega(t)\cl U-\cl U\|_{X^1_\Omega}+\|\cl U- U\|_{X^1_\Omega}\\
\leq (M+1)\rho+\sup_{t\in[0,\tau]}\|T_\Omega(t)\cl U-\cl U\|_{X^1_\Omega}\leq D/3+D/3=2D/3;
\end{multline}
For $\beta\geq\bar\beta$, $U\in X^1_\beta$ with $\|U-\cl U\|_{X^1_\beta}\leq\rho$ and $t\in[0,\tau]$ we have
\begin{multline}
\|T_\beta(t) U- U\|_{X^1_\beta}\\
\leq\|T_\beta(t)U-T_\beta(t)\cl U\|_{X^1_\beta}+\|T_\beta(t)\cl U-T_\Omega(t)\cl U\|_{X^1_\beta} +\|T_\Omega(t)\cl U-\cl U\|_{X^1_\beta}+\|\cl U- U\|_{X^1_\beta}\\
\leq (M+1)\|U-\cl U\|_{X^1_\beta}+\sup_{t\in[0,\tau]}\|T_\beta(t)\cl U-T_\Omega(t)\cl U\|_{X^1_\beta} +\sup_{t\in[0,T]}\|T_\Omega(t)\cl U-\cl U\|_{X^1_\Omega}\\
\leq (M+1)\rho+D/3+D/3\leq D.
\end{multline}
The proof is complete.
\end{proof}


\section{The nonlinear problem}

In this section we assume that $N=3$.
\begin{Hyp}\label{nonlinhypoth1}
Let $g\colon\R^N\times\R \to\R$ satisfy the following assumptions:
\begin{enumerate}
\item for every $s\in \R$ the function $g(\cdot, s)$ is measurable;
\item $g(\cdot, 0)=0$ almost everywhere;
\item for almost every $x\in\R^N$ the function $g(x,\cdot)$ is of class $C^1$;
\item there exists a constant $C_1>0$ and a non-negative measurable function $a(\cdot)$ such that for almost every $x\in\R^N$
\begin{equation*}|\partial_s g(x,s)|\leq C_1(a(x)+|s|^2), \quad s\in\R; \end{equation*}
\item the function $a(\cdot)$ in (4) is such that the assignment $u(\cdot)\mapsto a(\cdot)u(\cdot)$ defines a bounded linear map from $H^1(\R^3)$ to $L^2(\R^3)$.
\end{enumerate}
\end{Hyp}

\begin{Rem}
Due to the Sobolev embedding $H^1(\R^3)\hookrightarrow L^6(\R^3)$, a condition on $a(\cdot)$ ensuring that (5) is satisfied is the following: there exist $a_1(\cdot)\in L^{p_1}(\R^3)$, \dots, $a_m(\cdot)\in L^{p_m}(\R^3)$, with $3\leq p_j\leq+\infty$, $j=1$, \dots, $m$, such that $a=a_1+\cdots+ a_m$.
\end{Rem}

\begin{Hyp} \label{nonlinhypoth2}
For $\beta\geq 0$ let $\chi_\beta(\cdot)\in L^2(\R^3)$; moreover, let $\chi_\Omega(\cdot)\in L^2_\Omega(\R^3)$. We assume that
\begin{equation*}
\|\chi_\beta-\chi_\Omega\|_{L^2(\R^3)}\to 0\quad\text{as $\beta\to+\infty$.}
\end{equation*}
\end{Hyp}

We set 

\begin{equation}f_\beta(x,s):=\chi_\beta(x)+g(x,s)\quad\text{ and}\quad f_\Omega(x,s):=\chi_\Omega(x)+g(x,s).\end{equation}

We introduce the Nemitski operators associated with $f_\beta$ and $f_\Omega$. If $u\colon\R^3\to \R$, $\hat f_\beta(u)\colon\R^3\to\R$
and $\hat f_\Omega(u)\colon\R^3\to\R$ are defined by
\begin{equation}
\hat f_\beta(u) (x):=f_\beta(x,u(x))\quad\text{and}\quad\hat f_\Omega(u) (x):=f_\Omega(x,u(x)).
\end{equation}

We have the following result:

\begin{Prop} Assume that Hypotheses \ref{nonlinhypoth1} and \ref{nonlinhypoth2} are satisfied. Then 
 $\hat f_\beta$ maps $H^1(\R^3)$ into $L^2(\R^3)$, $\hat f_\Omega$ maps $H^1_\Omega(\R^3)$ into $L^2_\Omega(\R^3)$ and there exists a positive constant $C$
such that
\begin{enumerate}
\item for every $\beta\geq 0$ and every $u\in H^1(\R^3)$ \begin{equation*}\|\hat f_\beta(u)\|_{L^2(\R^3)}\leq C(1+\|u\|_{1,\beta}^3);\end{equation*}
\item for every $\beta\geq 0$ and every $u_1,u_2\in H^1(\R^3)$ \begin{equation*}\|\hat f_\beta(u_1)-\hat f_\beta(u_2)\|_{L^2(\R^3)}\leq C(1+\|u_1\|_{1,\beta}^2+\|u_2\|_{1,\beta}^2)\|u_1-u_2\|_{1,\beta};\end{equation*}
\item for every $u\in H^1_\Omega(\R^3)$ \begin{equation*} \|\hat f_\Omega(u)\|_{L^2_\Omega(\R^3)}\leq C(1+\|u\|_{H^1_\Omega(\R^3)}^3);\end{equation*}
\item for every $u_1,u_2\in H^1_\Omega(\R^3)$ \begin{equation*}\|\hat f_\Omega(u_1)-\hat f_\Omega(u_2)\|_{L^2_\Omega(\R^3)}\leq C(1+\|u_1\|_{H^1_\Omega(\R^3)}^2+\|u_2\|_{H^1_\Omega(\R^3)}^2)\|u_1-u_2\|_{H^1_\Omega(\R^3)};\end{equation*}
\item for every $u\in H^1_\Omega (\R^3)$ \begin{equation*}\|\hat f_\beta(u)-\hat f_\Omega(u)\|_{L^2(\R^3)}=\|\chi_\beta-\chi_\Omega\|_{L^2(\R^3)}\to 0\quad\text{as $\beta\to+\infty$}.\end{equation*}
\end{enumerate}
The constant $C$ depends only on $C_1$, on the norm of the operator  $u\mapsto au$ from $H^1(\R^3)$ to $L^2(\R^3)$, on $\sup_{\beta\geq 0}\|\chi_\beta\|_{L^2(\R^3)}$ and on the constant of the Sobolev embedding $H^1(\R^3)\hookrightarrow L^6(\R^3)$. In particular $C$ is independent of $\beta$.
\end{Prop}
\begin{proof}The statements are a straightforward consequence of Hypotheses \ref{nonlinhypoth1} and \ref{nonlinhypoth2}, of the Sobolev embedding $H^1(\R^3)\hookrightarrow L^6(\R^3)$ and of H\umlaut older inequality. \end{proof}

We consider the semilinear equation
\begin{equation}
\begin{aligned}
u_{tt}+\gamma u_t+V_\beta(x)u-\Delta u&=f_\beta(x,u),&&(t,x)\in[0,+\infty[\times\R^3,
\end{aligned}\end{equation}
which can be rewritten first as a system
\begin{equation}
\begin{cases} u_t=v\\
v_t=\Delta U-V_\beta(x)u-\gamma v+f_\beta(x,u)
\end{cases}
\end{equation}
and then as an abstract equation 
\begin{equation}\label{abstr1}
\dot U(t)=B_\beta U(t)+\Phi_\beta(U(t)), \quad U(t)=(u(t),v(t))\in X^1_\beta,
\end{equation}
where
\begin{equation}
\Phi_\beta(U):=(0,\hat f_\beta(u)),\quad U=(u,v)\in X^1_\beta.
\end{equation}
Similarly, the equation
\begin{equation}
\begin{aligned}
u_{tt}+\gamma u_t+u-\Delta u&=f_\Omega(x,u),&&(t,x)\in[0,+\infty[\times\Omega\\u&=0,&&(t,x)\in[0,+\infty[\times\partial\Omega
\end{aligned}\end{equation}
can be rewritten first as a system
\begin{equation}
\begin{cases} u_t=v\\
v_t=\Delta U-Vu-\gamma v+f_\Omega(x,u)
\end{cases}
\end{equation}
and then as an abstract equation 
\begin{equation}\label{abstr2}
\dot U(t)=B_\Omega U(t)+\Phi_\Omega(U(t)), \quad U(t)=(u(t),v(t))\in X^1_\Omega,
\end{equation}
where
\begin{equation}
\Phi_\Omega(U):=(0,\hat f_\Omega(u)),\quad U=(u,v)\in X^1_\Omega.
\end{equation}
We say that a continuous function $U\colon[0,\tau]\to X^1_\beta$ is a {\it mild solution} of (\ref{abstr1}) with initial datum $U_0\in X^1_\beta$ iff
\begin{equation}
U(t)=T_\beta(t)U_0+\int_0^t T_\beta(t-s)\Phi_\beta(U(s))\,ds, \quad t\in[0,\tau].
\end{equation}
Similarly, we say that a continuous function $U\colon[0,\tau]\to X^1_\Omega$ is a {\it mild solution} of (\ref{abstr2}) with initial datum $U_0\in X^1_\Omega$ iff
\begin{equation}
U(t)=T_\Omega(t)U_0+\int_0^t T_\Omega(t-s)\Phi_\Omega(U(s))\,ds, \quad t\in[0,\tau].
\end{equation}
The local existence and uniqueness of mild solutions for abstract equations like (\ref{abstr1}) and (\ref{abstr2}) is a well established fact  and is obtained by mean of the Banach contraction theorem (see e.g. \cite{CH}). However, since our aim is to pass to the (singular) limit as $\beta\to+\infty$, we need to go into the details of the proof of the existence result, in order to track the involved constants and control that they remain uniformly bounded as $\beta\to+\infty$.

We have the following {\it local} singular convergence result:

\begin{Theorem}\label{localnonlinear}
Assume Hypotheses \ref{prop_V}, \ref{nonlinhypoth1} and \ref{nonlinhypoth2} are satisfied.
Let $\mathcal K$ be a compact subset of $X^1_\Omega$, with $\|U\|_{X^1_\Omega}\leq R$ for all $U\in\mathcal K$. There exist $0<\rho\leq R$, $\tau>0$ and $\cl\beta>0$ such that if $\cl U_\Omega\in\mathcal K$ then
\begin{enumerate}
\item for every $U_{\Omega,0}\in X^1_\Omega$ with $\|U_{\Omega,0}-\cl U_\Omega\|_{X^1_\Omega}\leq\rho$ there is a unique mild solution $U_\Omega(t)$ of (\ref{abstr2}) on $[0,\tau]$ with $U_\Omega(0)=U_{\Omega,0}$, and 
\begin{equation*}
\sup_{t\in[0,\tau]}\|U_\Omega(t)-U_{\Omega,0}\|_{X^1_\Omega}\leq R;\end{equation*}
\item for every $\beta\geq\cl\beta$ and for every $U_{\beta,0}\in X^1_\beta$ with $\|U_{\beta,0}-\cl U_\Omega\|_{X^1_\beta}\leq\rho$ there is a unique mild solution $U_\beta(t)$ of (\ref{abstr1}) on $[0,\tau]$ with $U_\beta(0)=U_{\beta,0}$, and 
\begin{equation*}
\sup_{t\in[0,\tau]}\|U_\beta(t)-U_{\beta,0}\|_{X^1_\beta}\leq R;
\end{equation*}
\item let $(\beta_n)_{n\in\N}$ be a sequence of positive real numbers with $\beta_n\to +\infty$ as $n\to\infty$, let $U_{\Omega,0}\in X^1_\Omega$ with $\|U_{\Omega,0}-\cl U_\Omega\|_{X^1_\Omega}\leq \rho/2$, and for every $n\in\N$ let $U_{\beta_n,0}\in X^1_{\beta_n}$ with $\|U_{\beta_n,0}-U_{\Omega,0}\|_{X^1_{\beta_n}}\to 0$ as $n\to\infty$. Then there exist $\cl n\in\N$ such that for all $n\geq\cl n$ the solution $U_{\beta_n}(t)$ of (\ref{abstr1})  with $U_{\beta_n}(0)=U_{\beta_n,0}$
is defined $[0,\tau]$ and
\begin{equation}
\sup_{t\in[0,\tau]}\|U_{\beta_n}(t)-U_\Omega(t)\|_{X^1_{\beta_n}}
\to 0\quad\text{as $n\to\infty$.}
\end{equation}
\end{enumerate}
\end{Theorem}
\begin{proof}
First we observe that if $U=(u,v)\in X^1_\Omega$ with $\|U\|_{X^1_\Omega}\leq 3R$, then 
\begin{multline}
\|\Phi_\Omega(U)\|_{X^1_\Omega}=\|(0,\hat f_\Omega(u))\|_{X^1_\Omega}=\|\hat f_\Omega(u)\|_{L^2_\Omega(\R^3)}\\
\leq C(1+\|u\|_{H^1_\Omega(\R^3)}^3)\leq C(1+\|U\|_{X^1_\Omega}^3)\leq C(1+(3R)^3)=:K_{3R}.
\end{multline}
Moreover, if $U_1$ and $U_2\in X^1_\Omega$ with $\|U_1\|_{X^1_\Omega}\leq 3R$ and $\|U_2\|_{X^1_\Omega}\leq 3R$, then
\begin{multline}
\|\Phi_\Omega(U_1)-\Phi_\Omega(U_2)\|_{X^1_\Omega}
\leq C(1+\|U_1\|_{X^1_\Omega}^2+\|U_2\|_{X^1_\Omega}^2)
\|U_1-U_2\|_{X^1_\Omega}\\
\leq C(1+2(3R)^2)\|U_1-U_2\|_{X^1_\Omega}=:L_{3R}\|U_1-U_2\|_{X^1_\Omega}.
\end{multline}
Similarly, one can easily see that if $\beta\geq 0$ and $U=(u,v)\in X^1_\beta$ with $\|U\|_{X^1_\beta}\leq 3R$, then 
\begin{equation}
\|\Phi_\beta(U)\|_{X^1_\beta}\leq K_{3R},
\end{equation}
and if $U_1$ and $U_2\in X^1_\beta$ with $\|U_1\|_{X^1_\beta}\leq 3R$ and $\|U_2\|_{X^1_\beta}\leq 3R$, then
\begin{equation}
\|\Phi_\beta(U_1)-\Phi_\beta(U_2)\|_{X^1_\beta}
\leq L_{3R}\|U_1-U_2\|_{X^1_\beta}.
\end{equation}
Now we take $\rho>0$, $\tau>0$ and $\cl\beta>0$ such that $\rho\leq R$ and the conclusions of Corollary \ref{corolxyz} hold with $D=3R/4$. Namely, we ask that:
\begin{itemize}
\item $\rho>0$ is such that \begin{equation}\label{zzz1}(M+1)\rho\leq R/4;\end{equation}
\item $\tau>0$ is such that
\begin{equation}\label{zzz2}\sup_{\cl U\in\mathcal K}\sup_{t\in[0,\tau]}\|T_\Omega(t)\cl U-\cl U\|_{X^1_\Omega}
\leq R/4;
\end{equation}
\item $\cl\beta>0$ is such that for all $\beta\geq\cl\beta$
\begin{equation}\label{zzz3}\sup_{\cl U\in\mathcal K}\sup_{t\in[0,\tau]}\|T_\beta(t)\cl U-T_\Omega(t)\cl U\|_{X^1_\beta}
\leq R/4.
\end{equation}
\end{itemize}
Moreover, we ask that 
\begin{itemize}
\item $\tau>0$ is such that \begin{equation}\label{zzz4} M\tau K_{3R}\leq R/4\quad\text{and}\quad M\tau L_{3R}\leq 1/2.\end{equation}
\end{itemize}
Now let $\cl U_\Omega\in\mathcal K$ be fixed. Let $U_{\Omega,0}\in X^1_\Omega$ be such that $\|U_{\Omega,0}-\cl U_\Omega\|_{X^1_\Omega}\leq \rho$. We introduce the complete metric space
\begin{equation}
\mathscr{X}_{U_{\Omega,0}}:=\left\{ U:[0,\tau]\to X^1_\Omega\mid \text{$U$ is continuous and $\|U(t)-U_{\Omega,0}\|_{X^1_\Omega}\leq R$ for $t\in[0,\tau]$} \right\}
\end{equation}
equipped with the distance induced by the $C^0$-norm. We notice that  if $U\in\mathscr{X}_{U_{\Omega,0}} $ then $\|U(t)\|_{X^1_\Omega}\leq 3R$ for $t\in[0,\tau]$.
We define the map 

\begin{equation}\begin{aligned}&{\rm G}_{U_{\Omega,0}}\colon\mathscr{X}_{U_{\Omega,0}}\to C^0([0,\tau], X^1_\Omega)\\
&({\rm G}_{U_{\Omega,0}}U)(t):=T_\Omega(t)U_{\Omega,0}+\int_0^tT_\Omega(t-s)\Phi_\Omega(U(s))\,ds
\end{aligned}\end{equation}
If $U(\cdot)$ is a fixed point of ${\rm G}_{U_{\Omega,0}}$, then $U(\cdot)$ is a mild solution of (\ref{abstr2}) on $[0,\tau]$ with $U(0)=U_{\Omega,0}$. Let $U\in\mathscr{X}_{U_{\Omega,0}}$. For $t\in[0,\tau]$, by (\ref{zzz1}), (\ref{zzz2}), (\ref{zzz3}), (\ref{zzz4}) and Corollary \ref{corolxyz} we have:
\begin{multline}
\|({\rm G}_{U_{\Omega,0}}U)(t)-U_{\Omega,0}\|_{X^1_\Omega}\\
\leq \|T_\Omega(t)U_{\Omega,0}-U_{\Omega,0}\|_{X^1_\Omega}
+\int_0^t M \|\Phi_\Omega(U(s))\|_{X^1_\Omega}\,ds\\
\leq  \|T_\Omega(t)U_{\Omega,0}-U_{\Omega,0}\|_{X^1_\Omega}+\int_0^t M K_{3R}\,ds\\
\leq \|T_\Omega(t)U_{\Omega,0}-U_{\Omega,0}\|_{X^1_\Omega}+M\tau K_{3R}\leq \frac34 R+\frac14 R=R.
\end{multline}
Therefore ${\rm G}_{U_{\Omega,0}}$ maps $\mathscr{X}_{U_{\Omega,0}}$ into itself. Next, if $U_1$ and $U_2\in\mathscr{X}_{U_{\Omega,0}}$, then for $t\in[0,\tau]$ we have:
\begin{multline}
\|({\rm G}_{U_{\Omega,0}}U_1)(t)-({\rm G}_{U_{\Omega,0}}U_2)(t)\|_{X^1_\Omega}\\
\leq 
\int_0^t M \|\Phi_\Omega(U_1(s))-\Phi_\Omega(U_2(s))\|_{X^1_\Omega}\,ds
\leq  \int_0^t M L_{3R}\|U_1(s)-U_2(s))\|_{X^1_\Omega}\,ds\\
\leq M\tau L_{3R}\sup_{s\in[0,\tau]}\|U_1(s)-U_2(s))\|_{X^1_\Omega}  
\leq \frac12\sup_{s\in[0,\tau]}\|U_1(s)-U_2(s))\|_{X^1_\Omega}.
\end{multline}
We have just proved that ${\rm G}_{U_{\Omega,0}}$ is a contraction in the complete metric space $\mathscr{X}_{U_{\Omega,0}}$, so it has a unique fixed point. This proves the first statement of the Theorem.

In order to prove the second statement we procede in a similar way. Let $\beta\geq\cl\beta$ and let $U_{\beta,0}\in X^1_\beta$ be such that $\|U_{\beta,0}-\cl U_\Omega\|_{X^1_\beta}\leq \rho$. We introduce the complete metric space
\begin{equation}
\mathscr{X}_{U_{\beta,0}}:=\left\{ U:[0,\tau]\to X^1_\beta\mid \text{$U$ is continuous and $\|U(t)-U_{\beta,0}\|_{X^1_\beta}\leq R$ for $t\in[0,\tau]$} \right\}
\end{equation}
equipped with the distance induced by the $C^0$-norm. If $U\in\mathscr{X}_{U_{\beta,0}} $ then $\|U(t)\|_{X^1_\beta}\leq 3R$ for $t\in[0,\tau]$.
We define the map 

\begin{equation}\begin{aligned}&{\rm G}_{U_{\beta,0}}\colon\mathscr{X}_{U_{\beta,0}}\to C^0([0,\tau], X^1_\beta)\\
&({\rm G}_{U_{\beta,0}}U)(t):=T_\beta(t)U_{\beta,0}+\int_0^tT_\beta(t-s)\Phi_\beta(U(s))\,ds
\end{aligned}\end{equation}
Again, if $U(\cdot)$ is a fixed point of ${\rm G}_{U_{\beta,0}}$, then $U(\cdot)$ is a mild solution of (\ref{abstr1}) on $[0,\tau]$ with $U(0)=U_{\beta,0}$. Let $U\in\mathscr{X}_{U_{\beta,0}}$. For $t\in[0,\tau]$, by (\ref{zzz1}), (\ref{zzz2}), (\ref{zzz3}), (\ref{zzz4}) and Corollary \ref{corolxyz} we have:
\begin{multline}
\|({\rm G}_{U_{\beta,0}}U)(t)-U_{\beta,0}\|_{X^1_\beta}\\
\leq \|T_\beta(t)U_{\beta,0}-U_{\beta,0}\|_{X^1_\beta}
+\int_0^t M \|\Phi_\beta(U(s))\|_{X^1_\beta}\,ds\\
\leq  \|T_\beta(t)U_{\beta,0}-U_{\beta,0}\|_{X^1_\beta}+\int_0^t M K_{3R}\,ds\\
\leq \|T_\beta(t)U_{\beta,0}-U_{\beta,0}\|_{X^1_\beta}+M\tau K_{3R}\leq \frac34 R+\frac14 R=R.
\end{multline}
This means that ${\rm G}_{U_{\beta,0}}$ maps $\mathscr{X}_{U_{\beta,0}}$ into itself. Next, if $U_1$ and $U_2\in\mathscr{X}_{U_{\beta,0}}$, then for $t\in[0,\tau]$ we have:
\begin{multline}
\|({\rm G}_{U_{\beta,0}}U_1)(t)-({\rm G}_{U_{\beta,0}}U_2)(t)\|_{X^1_\beta}\\
\leq 
\int_0^t M \|\Phi_\beta(U_1(s))-\Phi_\beta(U_2(s))\|_{X^1_\beta}\,ds
\leq  \int_0^t M L_{3R}\|U_1(s)-U_2(s))\|_{X^1_\beta}\,ds\\
\leq M\tau L_{3R}\sup_{s\in[0,\tau]}\|U_1(s)-U_2(s))\|_{X^1_\beta}  
\leq \frac12\sup_{s\in[0,\tau]}\|U_1(s)-U_2(s))\|_{X^1_\beta}.
\end{multline}
This means that ${\rm G}_{U_{\beta,0}}$ is a contraction in the complete metric space $\mathscr{X}_{U_{\beta,0}}$, and therefore it has a unique fixed point. This concludes the proof of the second statement of the Theorem.

Now we move on to the proof of the third statement. Let $\beta_n\to +\infty$ as $n\to\infty$, let $U_{\Omega,0}\in X^1_\Omega$ with $\|U_{\Omega,0}-\cl U_\Omega\|_{X^1_\Omega}\leq \rho/2$, and for every $n\in\N$ let $U_{\beta_n,0}\in X^1_{\beta_n}$ with $\|U_{\beta_n,0}-U_{\Omega,0}\|_{X^1_{\beta_n}}\to 0$ as $n\to\infty$. Then there exists $\cl n\in\N$ such that for all $n\geq\cl n$ one has $\beta_n\geq\cl\beta$ and  $\|U_{\beta_n,0}-U_{\Omega,0}\|_{X^1_{\beta_n}}\leq\rho/2$, and consequently $\|U_{\beta_n,0}-\cl U_\Omega\|_{X^1_{\beta_n}}\leq\rho$. By parts (1) and (2) of the present theorem, there is a unique mild solution $U_\Omega(\cdot)\in\mathscr{X}_{U_{\Omega,0}}$ of (\ref{abstr2}) with $U_\Omega(0)=U_{\Omega,0}$; moreover, for all $n\geq\cl n$
there is a unique mild solution $U_{\beta_n}(\cdot)\in\mathscr{X}_{U_{\beta_n,0}}$ of (\ref{abstr1}) with $U_{\beta_n}(0)=U_{\beta_n,0}$. More explicitely, we have
\begin{equation}
U_{\Omega}(t)=T_\Omega(t)U_{\Omega,0}+\int_0^tT_\Omega(t-s)\Phi_\Omega(U_{\Omega}(s))\,ds,\quad t\in[0,\tau]
\end{equation}
and
\begin{equation}
U_{\beta_n}(t)=T_{\beta_n}(t)U_{\beta_n,0}+\int_0^tT_{\beta_n}(t-s)\Phi_{\beta_n}(U_{\beta_n}(s))\,ds,\quad t\in[0,\tau].
\end{equation}
Then for $t\in[0,\tau]$ we have:
\begin{multline}
\|U_{\beta_n}(t)-U_\Omega(t)\|_{X^1_{\beta_n}}\leq\|T_{\beta_n}(t)U_{\beta_n,0}-T_\Omega(t)U_{\Omega,0}\|_{X^1_{\beta_n}}\\
+\|\int_0^tT_{\beta_n}(t-s)\left(\Phi_{\beta_n}(U_{\beta_n}(s)-\Phi_{\beta_n}(U_{\Omega}(s))\right)\,ds\|_{X^1_{\beta_n}}\\
+\|\int_0^tT_{\beta_n}(t-s)\left(\Phi_{\beta_n}(U_{\Omega}(s)-\Phi_{\Omega}(U_{\Omega}(s))\right)\,ds\|_{X^1_{\beta_n}}\\
+\|\int_0^t\left(T_{\beta_n}(t-s)-T_{\Omega}(t-s)\right)\Phi_{\Omega}(U_{\Omega}(s))\,ds\|_{X^1_{\beta_n}}.
\end{multline}
It follows that
\begin{multline}
\|U_{\beta_n}(t)-U_\Omega(t)\|_{X^1_{\beta_n}}\leq\sup_{s\in[0,\tau]}\|T_{\beta_n}(s)U_{\beta_n,0}-T_\Omega(s)U_{\Omega,0}\|_{X^1_{\beta_n}}\\
+M\tau L_{3R}\sup_{s\in[0,\tau]}\|U_{\beta_n}(s)-U_\Omega(s)\|_{X^1_{\beta_n}}+M\tau\|\chi_{\beta_n}-\chi_\Omega\|_{L^2(\R^3)}\\
+\tau\sup_{0\leq s\leq t\leq\tau}\|\left(T_{\beta_n}(t-s)-T_{\Omega}(t-s)\right)\Phi_{\Omega}(U_{\Omega}(s))\|_{X^1_{\beta_n}}.
\end{multline}
Since $M\tau L_{3R}\leq1/2$, we have
\begin{multline}
\sup_{t_\in[0,\tau]}\|U_{\beta_n}(t)-U_\Omega(t)\|_{X^1_{\beta_n}}\\
\leq2\sup_{s\in[0,\tau]}\|T_{\beta_n}(s)U_{\beta_n,0}-T_\Omega(s)U_{\Omega,0}\|_{X^1_{\beta_n}}
+2M\tau\|\chi_{\beta_n}-\chi_\Omega\|_{L^2(\R^3)}
\\+2\tau\sup_{0\leq s\leq t\leq\tau}\|\left(T_{\beta_n}(t-s)-T_{\Omega}(t-s)\right)\Phi_{\Omega}(U_{\Omega}(s))\|_{X^1_{\beta_n}}
=:a_n+b_n+c_n.
\end{multline}
Now we have that $a_n\to0$ as $n\to\infty$ by Theorem \ref{TrKa}, and $b_n\to 0$ as $n\to\infty$ by Hypothesis \ref{nonlinhypoth2}.
Concerning $(c_n)_{n,\in\N}$, we apply the result of
Corollary \ref{uniformTrKa} to the compact set $\tilde{\mathcal K}:=\{\Phi_\Omega(U_\Omega(s))\mid{s\in[0,\tau]}\}\subset X^1_\Omega$ and we get that $c_n\to0$ as $n\to\infty$. The proof is complete.
\end{proof}

Finally, we have the following {\it global} singular convergence result:
\begin{Theorem}\label{globalnonlinear} Assume Hypotheses \ref{prop_V}, \ref{nonlinhypoth1} and \ref{nonlinhypoth2} are satisfied. Let $ U_{\Omega,0}\in X^1_\Omega$, let $T>0$, and suppose that there exists a mild solution $U_\Omega\colon[0,T]\to X^1_\Omega$ of (\ref{abstr2}), with $U_\Omega(0)= U_{\Omega,0}$.
Let $(\beta_n)_{n\in\N}$ be a sequence of positive numbers with $\beta_n\to+\infty$ as $n\to\infty$, and for every $n\in\N$ let $ U_{\beta_n,0}\in X^1_{\beta_n}$. Assume that $\| U_{\beta_n,0}- U_{\Omega,0}\|_{X^1_{\beta_n}}\to0$ as $n\to\infty$. Then there is $\cl n\in\N$ such that for all $n\geq\cl n$  there exists a mild solution $U_{\beta_n}\colon[0,T]\to X^1_{\beta_n}$ of (\ref{abstr1}), with $U_{\beta_n}(0)= U_{\beta_n,0}$, and
\begin{equation}
\sup_{t\in[0,T]}\|U_{\beta_n}(t)-U_\Omega(t)\|_{X^1_{\beta_n}}\to 0\quad\text{as $n\to\infty$.}
\end{equation}
\end{Theorem}
\begin{proof}
Define the compact subset $\mathcal K$ of $X^1_\Omega$:
\begin{equation}
\mathcal K:=\{U_\Omega(t)\mid t\in[0,T]\}.
\end{equation}
Let $R>0$ be such that $\|U\|_{X^1_\Omega}\leq R$ for all $U\in\mathcal K$. The proof is obtained by applying the results of Theorem \ref{localnonlinear} a finite number of times.
\end{proof}

We conclude with a remark concerning exponential dichotomies. Let $\nu>0$ and consider the operators 
$$
B_{\Omega, \nu}(u,v)=(v,-(A_\Omega -\nu I) u-\gamma v)=B_\Omega(u,v)+Z_\nu(0, u)
$$ 
and 
$$B_{\beta, \nu}(u,v)=(v,-(A_\beta -\nu I) u-\gamma v)=B_\beta(u,v)+Z_\nu(0, u),$$
where $Z_\nu(U)=Z_\nu(u,v)=(0,-\nu u)$. 
It follows from \cite[Th. 6.4]{Gold} that $B_{\Omega, \nu}$ and $B_{\beta, \nu}$ generate strongly continuous semigroups $T_{\Omega,\nu}(t)$ and $T_{\beta,\nu}(t)$ in $X^1_\Omega$ and $X^1_\beta$ respectively. Moreover, the following identities hold:
$$
T_{\Omega,\nu}(t)U=T_\Omega(t)U+\int_0^t T_\Omega(t-s)Z_\nu(T_{\Omega,\nu}(s)U)\,ds,\quad U\in X^1_\Omega
$$
and
$$
T_{\beta,\nu}(t)U=T_\beta(t)U+\int_0^t T_\beta(t-s)Z_\nu(T_{\beta,\nu}(s)U)\,ds,\quad U\in X^1_\beta.
$$

 Let $(\beta_n)_{n\in\N}$ be a sequence of positive numbers, $\beta_n\to+\infty$ as $n\to\infty$. Let $U_n=(u_n,v_n)\in X^1_{\beta_n}$ for $n\in\N$, let $U=(u,v)\in X^1_\Omega$,  and assume that $\|U_n -U\|_{X^1_{\beta_n}}\to 0$ as $n\to\infty$. Then Theorem \ref{globalnonlinear}  implies that 
 \begin{equation}
 \label{blabla}\|T_{\beta_n,\nu}(t)U_n-T_{\Omega,\nu}(t)U\|_{X^1_{\beta_n}}\to 0 \quad\text{as $n\to\infty$.}
 \end{equation}

 Assume now that $\gamma>0$. If $\nu>0$ is large and is not an eigenvalue of $A_\Omega$, it is easy to see that the semigroup $T_{\Omega,\nu}(t)$ has a nontrivial exponential dichotomy, with spectral projection $Q_{\Omega,\nu}$. 
Thanks to the spectral convergence results of Section 3, also the semigroups $T_{\beta,\nu}(t)$ have an exponential dichotomy, with spectral projections $Q_{\beta,\nu}$, provided $\beta$ is sufficiently large. Moreover, it is not difficult to show that $\|Q_{\beta_n,\nu}U_n-Q_{\Omega,\nu}U\|_{X^1_{\beta_n}}\to 0$ as $n\to\infty$.
This fact, together with (\ref{blabla}), implies that 
$$\|T_{\beta_n,\nu}(t)Q_{\beta_n,\nu}U_n-T_{\Omega,\nu}(t)Q_{\Omega,\nu}U\|_{X^1_{\beta_n}}\to 0 \quad\text{as $n\to\infty$.}$$

\end{document}